\documentclass{article}
\usepackage[utf8]{inputenc}

\usepackage{algorithm}
\usepackage{algorithmic}
\usepackage{amsmath}
\usepackage{amssymb}  
\usepackage{amsthm}
\usepackage{arydshln}
\usepackage{bbold}
\usepackage{bm}
\usepackage{booktabs}
\usepackage{bookmark}
\usepackage{calc}              % support command '\widthof'
\usepackage{extarrows}
\usepackage{float}
\usepackage{geometry}
\usepackage{graphicx}          % Required for inserting images
\usepackage{latexsym}
\usepackage{mathtools}
\usepackage{multirow}          % multirow in table
\usepackage{stmaryrd}
\usepackage{subcaption}
\usepackage{tikz}              % drawing the mesh
\usepackage{xcolor}
\usetikzlibrary{shapes.geometric, arrows, chains}

\usepackage{authblk}
\newtheorem{lemma}{Lemma}
\newtheorem{theorem}{Theorem}

\title{A residual driven multiscale method for Darcy's flow in perforated domains}
\date{}
\usepackage{indentfirst}
\author[a]{Wei Xie}
\author[b]{Shubin Fu$^*$}
\author[c]{Yin Yang$^*$}
\author[d]{Yunqing Huang}
\affil[a]{Hunan Key Laboratory for Computation and Simulation in Science and Engineering, National Center for Applied Mathematics in Hunan, Xiangtan University, Xiangtan, Hunan 41105, China}
\affil[b]{Eastern Institute for Advanced Study, Eastern Institute of Technology, Ningbo,  Zhejiang 315200, China}
\affil[c]{Hunan Research Center of the Basic Discipline Fundamental Algorithmic Theory and Novel Computational Methods, Key Laboratory of Intelligent Computing and Information Processing of Ministry of Education, Xiangtan University, Xiangtan, Hunan 41105, China}
\affil[d]{Hunan International Scientific and Technological Innovation Cooperation Base of Computational Science, Xiangtan University, Xiangtan, Hunan 411105, China}

\begin{document}

\maketitle
% \tableofcontents
\renewcommand{\thefootnote}{}%
\footnotetext{$^*$Corresponding authors: Shubin Fu, Yin Yang}
\footnotetext{E-mail addresses: xiew@smail.xtu.edu.cn (Wei Xie), 
shubinfu@eias.ac.cn (Shubin Fu),
yangyinxtu@xtu.edu.cn (Yin Yang),
huangyq@xtu.edu.cn (Yunqing Huang)}
\addtocounter{footnote}{-1}

\begin{abstract}
In this paper, we present a residual-driven multiscale method for simulating Darcy flow in perforated domains, where complex geometries and highly heterogeneous permeability make direct simulations computationally expensive. To address this, we introduce a velocity elimination technique that reformulates the mixed velocity-pressure system into a pressure-only formulation, significantly reducing complexity by focusing on the dominant pressure variable.
Our method is developed within the Generalized Multiscale Finite Element Method (GMsFEM) framework. For each coarse block, we construct offline basis functions from local spectral problems that capture key geometric and physical features. Online basis functions are then adaptively enriched using residuals, allowing the method to incorporate global effects such as source terms and boundary conditions, thereby improving accuracy.
We provide detailed error analysis demonstrating how the offline and online spaces contribute to the accuracy and efficiency of the solution. Numerical experiments confirm the method’s effectiveness, showing substantial reductions in computational cost while maintaining high accuracy—particularly through adaptive online enrichment. These results highlight the method’s potential for efficient and accurate simulation of Darcy flow in complex, heterogeneous perforated domains.
\end{abstract}

\section{Introduction}
Perforated materials, consisting of solid and fluid phases, are essential to many natural and industrial processes. They play a critical role in applications such as groundwater management, oil extraction, soil stabilization, and pollutant filtration, in which understanding flow and transport phenomena is crucial. However, their complex geometries and multiple perforation scales make accurate modeling and prediction challenging. Addressing these issues often requires formulating and solving mathematical models that capture the underlying physical processes.

Traditional numerical methods often struggle with this problem because of the high computational cost required to resolve fine-scale heterogeneities and complex geometries. To overcome these challenges, researchers have developed various specialized methods \cite{maalqvist2014localization, brown2016multiscale, huang2001partition, huang2022finite, huang2026pinn, zhai2024new}. One prominent approach is to use upscaling techniques, which incorporate microscopic information into coarse-scale models without explicitly resolving fine-scale features. Examples include the homogenization method \cite{lions1980asymptotic,allaire1991homogenizationCPAM}, the heterogeneous multiscale method \cite{henning2009heterogeneous, abdulle2012heterogeneous}, and the multicontinuum homogenization method \cite{efendiev2023multicontinuum, xie2025multicontinuum}. These methods offer different strategies for bridging micro- and macro-scales, typically by solving a series of local ``cell'' problems.

Motivated by the concept of homogenization, Hou proposed the Multiscale Finite Element Method (MsFEM) \cite{hou1997multiscale}, which involves solving local problems to construct special basis functions \cite{le2014msfem,muljadi2015nonconforming}. Compared to traditional finite element bases, these multiscale basis functions capture significantly more local detail. Building on this idea, many multiscale methods have been developed to further enrich the multiscale solution space or to mitigate artificial boundary effects through oversampling techniques. Examples of such methods include the Generalized Multiscale Finite Element Method (GMsFEM) \cite{chung2016generalized, chung2016mixed, chung2017online, chung2017conservative, alikhanov2025multiscale, xie2025multiscale}, the wavelet-based edge multiscale finite element method \cite{fu2019edge,fu2021wavelet}, and the Constraint Energy Minimizing Generalized Multiscale Finite Element Method (CEM-GMsFEM) \cite{chung2018constraint, chung2021convergence, xie2024cem}.

Another important aspect in multiscale modeling is the preservation of key physical quantities. This aspect is especially crucial in long-term simulations of subsurface flow, where conserving quantities such as mass ensures both accuracy and stability. The mixed finite element method is a commonly used technique to enforce these conservation properties, thereby ensuring an accurate representation of key physical quantities in numerical simulations.

The mixed generalized multiscale finite element method (MGMsFEM) \cite{chung2015mixed, yang2020online} has since been applied to numerous practical problems, including linear poroelasticity \cite{wang2021mixed}, Stokes flow \cite{wang2023locally}, multiphase transport \cite{wang2021comparison}, and flow in perforated domains \cite{chung2016mixed, chung2018multiscale}. 
In  \cite{chen2020generalized}, a velocity elimination technique was combined with trapezoidal integration  \cite{wheeler2006multipoint} in the lowest-order Raviart–Thomas (RT0) finite element space to diagonalize the mass matrix.
This reformulation allows the original saddle-point problem to be re-expressed with pressure as the only unknown variable. As a result, a systematic approach was proposed to construct multiscale basis functions that focus solely on the pressure variable, applied independently within disjoint coarse blocks 
  \cite{he2021adaptive, he2021generalized, he2024online}.
To further accelerate the solution process, efficient preconditioners \cite{fu2023efficient, fu2024adaptive, ye2024highly, ye2024robust, ye2024fast} were developed to fully exploit the problem’s disjoint, parallelizable structure. 

\begin{figure}[htbp]
\centering
\begin{tikzpicture}
\draw[fill=gray!80] (-1, -1) rectangle (8, 5);
\draw[fill=white] (6.7, 2.3) ellipse (0.5 and 0.7);
\draw[fill=white] (0, 2) ellipse (0.5 and 0.9);
\draw[fill=white] (4.5, 3) ellipse (0.8 and 0.5);
\draw[fill=white] (1.5, 1.5) circle (0.5);
\draw[fill=white] (3.5, 2) circle (0.15);
\draw[fill=white] (2.7, 1.2) circle (0.2);
\draw[fill=white] (4.5, 1.8) circle (0.12);
\draw[fill=white] (4, 1) circle (0.5);
\draw[fill=white] (0.5, 0.5) circle (0.18);
\draw[fill=white] (1.5, 0.5) circle (0.15);
\draw[fill=white] (1.25, 2.6) circle (0.15);
\draw[fill=white] (2.5, 3) circle (0.5);
\draw[fill=white] (5.5, 2) circle (0.12);
\draw[fill=white] (1.25, -0.1) circle (0.15);
\draw[fill=white] (2.5, 0) circle (0.25);
\draw[fill=white] (4, -0.15) circle (0.13);
\draw[fill=white] (5.5, -0.2) circle (0.12);
\draw[fill=white] (1.25, 4) circle (0.1);
\draw[fill=white] (2.5, 4.2) circle (0.15);
\draw[fill=white] (4, 3.75) circle (0.13);
\draw[fill=white] (5.5, 3.9) circle (0.22);
% \node at (0.5,3.5)   {$\Omega^{\epsilon}$};
\end{tikzpicture}
\caption{Illustration of a perforated domain $\Omega^{\epsilon}$.}
\label{fig:perforateddomain_example}
\end{figure}
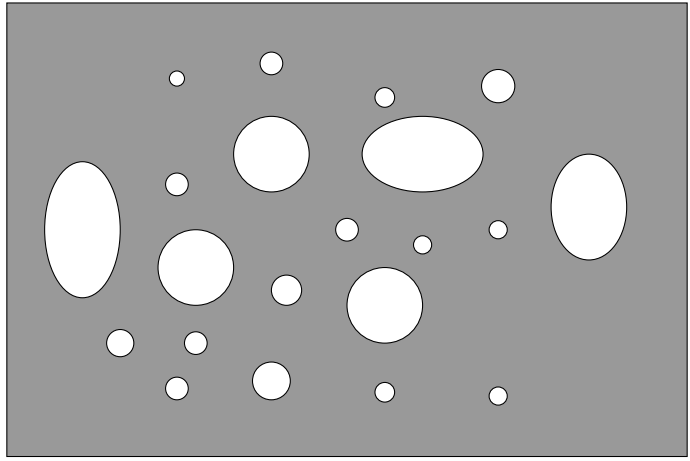

In this work, we propose a multiscale method specifically for Darcy flow in perforated domains (see Figure \ref{fig:perforateddomain_example}). By applying the trapezoidal integration technique \cite{wheeler2006multipoint}, we reformulate the mixed Darcy problem as a pressure-only system, which enables the construction of pressure-based multiscale basis functions. Our approach consists of two main stages: an offline stage and an online stage. In the offline stage, we solve a local spectral problem in each coarse block and select the eigenfunctions corresponding to the smallest eigenvalues to form the offline space. We define $\Lambda$  as the smallest eigenvalue whose associated eigenfunction is not included in this space. In the online stage, residual-based indicators are used to identify coarse blocks that require enrichment. Additional basis functions are then constructed locally for those blocks. Furthermore, we introduce an adaptive algorithm with a user-defined parameter $\theta$ to control the number of basis functions added at each iteration.
We provide a rigorous numerical analysis demonstrating that the error in the offline stage is bounded by $\Lambda$, whereas the online error depends on $\Lambda$, $\theta$, and the number of online iterations. These theoretical findings are confirmed by extensive numerical experiments, which also demonstrate the efficiency and robustness of the proposed method.

Our main contributions are:
\begin{enumerate}
    \item Development of a multiscale model reduction framework for Darcy flow in perforated domains.
    \item Formulation of local cell problems under various scenarios to enhance adaptability and accuracy.
    \item Detailed and rigorous numerical analysis of the proposed method.
    \item Comprehensive numerical experiments that validate the efficiency and robustness of the method.
\end{enumerate}

This paper is organized as follows.
In Section \ref{sec:preliminaries}, we present the model problem, the fine-scale discretization, and the process of reformulating the system into a pressure-only form.
The construction of offline and online basis functions is detailed in Sections \ref{sec:offline_stage} and \ref{sec:online_stage}, respectively.
In Section \ref{sec:convergence_results}, we provide the numerical analysis of the proposed method.
Section \ref{sec:numerical_results} demonstrates the numerical results.
Finally, the conclusion is provided in Section \ref{sec:conclusions}.

\section{Preliminaries} \label{sec:preliminaries}

In this section, we introduce the model problem, describe the fine-grid discretization, and present the velocity elimination technique used to transform the mixed saddle-point problem into a pressure-only algebraic system.

\subsection{Problem setup}

Let $\Omega \subset \mathbb{R}^2$ be a bounded convex domain. 
The perforated domain $\Omega^{\epsilon}$ is obtained by removing a set of perforations $\mathcal{B}^{\epsilon}$ from $\Omega$, as illustrated in Figure~\ref{fig:perforateddomain_example}. We consider the mixed formulation of the Darcy flow problem with heterogeneous permeability coefficients:

\begin{equation}
\begin{cases}
\begin{aligned}
\kappa^{-1} \mathbf{u} + \nabla p &= 0,
&\text{in } \Omega^{\epsilon}, \\
\mathrm{div}(\mathbf{u}) &= f,
&\text{in } \Omega^{\epsilon}.
\end{aligned}
\end{cases}
\label{eq:pde_mixed}
\end{equation}
Here, $\mathbf{u}$ denotes the Darcy velocity, $p$ the pressure, and $f$ the source term.
The permeability field $\kappa$ is assumed to be heterogeneous.
For simplicity, we impose a no-flux boundary condition $\mathbf{u} \cdot \mathbf{n} = 0$ on $\partial \Omega^{\epsilon}$, where $\mathbf{n}$ denotes the unit outward normal vector.

The Sobolev spaces for velocity and pressure are defined as:
\[
\mathbf{V} = \{\mathbf{v} \in [L^2(\Omega^{\epsilon})]^2 \mid \mathrm{div}(\mathbf{v}) \in L^2(\Omega^{\epsilon}),~ \mathbf{v} \cdot \mathbf{n} = 0~ \text{on } \partial \Omega^{\epsilon}\}, \quad Q = L^2(\Omega^{\epsilon}).
\]

The weak formulation of problem~\eqref{eq:pde_mixed} is to find $(\mathbf{u}, p) \in \mathbf{V} \times Q$ such that:
\begin{equation}
\begin{cases}
\begin{aligned}
\int_{\Omega^{\epsilon}} \kappa^{-1} \mathbf{u} \cdot \mathbf{v}
- \int_{\Omega^{\epsilon}} \mathrm{div}(\mathbf{v}) p &= 0, 
&\forall \mathbf{v} \in \mathbf{V}, \\
\int_{\Omega^{\epsilon}} \mathrm{div}(\mathbf{u}) q
&= \int_{\Omega^{\epsilon}} f q, 
&\forall q \in Q.
\end{aligned}
\end{cases}
\label{eq:weak_darcy}
\end{equation}

\subsection{Fine-grid discretization}

Let $\mathcal{T}_h$ denote a fine partition of $\Omega^{\epsilon}$ with mesh size $h$, sufficiently small to resolve both the geometry of $\mathcal{B}^{\epsilon}$ and the variations in $\kappa$.
While the current work focuses on rectangular grids, similar constructions extend naturally to triangular, general quadrilateral, and hexahedral meshes; see~\cite{wheeler2006multipoint,he2021generalized,hou2014numerical} for related developments.
The permeability $\kappa_w$ is assumed constant within each element $w \in \mathcal{T}_h$. Let $\mathcal{E}_h$ and $\mathcal{E}_h^0$ denote the sets of all fine-grid edges and interior edges, respectively.

Among the many mixed finite element methods (cf.~\cite{boffi2013mixed}), we use the lowest-order Raviart--Thomas element (RT0) due to its advantageous numerical integration properties. The finite element spaces are given by:
\[
\begin{aligned}
\mathbf{V}_h &= \{ \mathbf{v}_h \in \mathbf{V} \mid \mathbf{v}_h(w) = (\alpha_1 + \beta_1 x_1, \alpha_2 + \beta_2 x_2),\ \alpha_i, \beta_i \in \mathbb{R},~ w \in \mathcal{T}_h \}, \\
Q_h &= \{ q_h \in Q \mid q_h(w)~\text{is constant},~ w \in \mathcal{T}_h \}.
\end{aligned}
\]

Let $\{ \boldsymbol{\phi}_e \}_{e \in \mathcal{E}_h^0}$ and $\{ \varphi_w \}_{w \in \mathcal{T}_h}$ be the basis functions for $\mathbf{V}_h$ and $Q_h$, respectively. The discrete solution is expressed as:
\begin{equation}
\mathbf{u}_h = \sum_{e \in \mathcal{E}_h^0} u_e \boldsymbol{\phi}_e, \quad
p_h = \sum_{w \in \mathcal{T}_h} p_w \varphi_w.
\label{eq:uhph_express}
\end{equation}

Substituting~\eqref{eq:uhph_express} into~\eqref{eq:weak_darcy} leads to the discrete variational problem: find $(\mathbf{u}_h, p_h) \in \mathbf{V}_h \times Q_h$ such that
\begin{equation}
\begin{cases}
\begin{aligned}
\int_{\Omega^{\epsilon}} \kappa^{-1} \mathbf{u}_h \cdot \mathbf{v}_h 
- \int_{\Omega^{\epsilon}} \mathrm{div}(\mathbf{v}_h) p_h &= 0, 
&\forall \mathbf{v}_h \in \mathbf{V}_h, \\
\int_{\Omega^{\epsilon}} \mathrm{div}(\mathbf{u}_h) q_h 
&= \int_{\Omega^{\epsilon}} f q_h, &\forall q_h \in Q_h.
\end{aligned}
\end{cases}
\label{eq:weak_rt_fine}
\end{equation}

This formulation results in the following saddle-point system:
\begin{equation}
\begin{bmatrix}
M & -B \\
-B^T & \mathbf{0}
\end{bmatrix}
\begin{bmatrix}
\mathbf{u}_h \\ p_h
\end{bmatrix} =
\begin{bmatrix}
\mathbf{0} \\ -F
\end{bmatrix},
\label{eq:matrix_UPh}
\end{equation}
where $\mathbf{u}_h$ and $p_h$ are coefficient vectors corresponding to the velocity and pressure basis functions, respectively. 
For notational simplicity, we use the same symbols as for the continuous variables.
The matrix entries are defined as
\[
M_{ij} := \int_{\Omega^{\epsilon}} \kappa^{-1} \boldsymbol{\phi}_j \cdot \boldsymbol{\phi}_i, \quad
B_{ij} := \int_{\Omega^{\epsilon}} \varphi_j \, \mathrm{div}(\boldsymbol{\phi}_i), \quad
F_i := \int_{\Omega^{\epsilon}} f \varphi_i.
\]

We now describe how to eliminate the velocity unknowns using the trapezoidal integration rule~\cite{arbogast1997mixed,he2021adaptive,ye2024robust}.
From the matrix equation~\eqref{eq:matrix_UPh}, we have
\[
p_h = (B^T M^{-1} B)^{-1} F, \qquad \mathbf{u}_h = M^{-1} B p_h.
\]
This formulation is effective when:
(1) the inverse of the mass matrix $M$ is cheap to compute;
(2) the matrix $B^T M^{-1} B$ retains a sparse structure.

For a quadrilateral element $w \in \mathcal{T}_h$ with corner points $\mathbf{y}_i$, $1 \le i \le 4$, we apply the trapezoidal rule:
\begin{equation}
\int_{w} \kappa^{-1} \mathbf{u}_h \cdot \mathbf{v}_h \approx
\frac{|w|}{4} \sum_{i=1}^4 \kappa_w^{-1} (\mathbf{u}_h \cdot \mathbf{v}_h)(\mathbf{y}_i)
= \frac{|w|}{2} \sum_{e \in \partial w} \kappa_w^{-1} u_e v_e.
\label{eq:kuv_trapezoidal}
\end{equation}
This implies that the mass matrix $M$ is diagonal and can be inverted efficiently.
A more detailed derivation of this quadrature approximation can be found in~\cite{he2021adaptive,ye2024robust}.

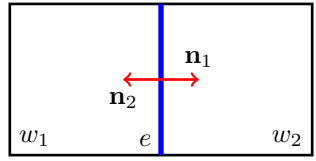
\begin{figure}[htbp]
\centering
\begin{tikzpicture}
\draw[line width=1pt] (0,0) rectangle (4,2);
\draw[blue, line width=2pt] (2,0) -- (2,2);
\draw[red, line width=1pt, ->] (2,1) -- (2.5,1);
\draw[red, line width=1pt, ->] (2,1) -- (1.5,1);
\node[anchor=south west] at (0,0)   {$w_1$};
\node[anchor=south east] at (4,0)   {$w_2$};
\node[anchor=south east] at (2,0)   {$e$};
\node[anchor=north] at (2.5,1.5)   {$\mathbf{n}_1$};
\node[anchor=south] at (1.5,0.5)   {$\mathbf{n}_2$};
\end{tikzpicture}
\caption{Two fine-grid elements, $w_1$ and $w_2$, sharing a common edge $e$.}
\label{fig:z1z2}
\end{figure}

We now show how to recover the velocity from the pressure. Let $\boldsymbol{\phi}$ be the RT0 basis function associated with edge $e$ (see Figure~\ref{fig:z1z2}). Then,
\[
\begin{aligned}
\int_{w_1 \cup w_2} \kappa^{-1} \mathbf{u}_h \cdot \boldsymbol{\phi}
&= \int_{w_1 \cup w_2} \mathrm{div}(\boldsymbol{\phi}) p_h, \\
\frac{1}{2} (|w_1|\kappa_1^{-1} + |w_2|\kappa_2^{-1}) u_e &= \int_{e} p_h \, \boldsymbol{\phi} \cdot (\mathbf{n}_1 + \mathbf{n}_2),
\end{aligned}
\]
where $\kappa_i$ and $\mathbf{n}_i$ denote the permeability and outward normal on $w_i$, respectively.
Since $\mathbf{n}_1 = -\mathbf{n}_2$, we obtain:
\begin{equation}
u_e = \bar{\kappa}_e |e| \, \llbracket p_h \rrbracket_e,
\label{eq:veq_darcylaw}
\end{equation}
where $\bar{\kappa}_e = 2/(|w_1|/\kappa_1 + |w_2|/\kappa_2)$ and $\llbracket p_h \rrbracket_e = p_{w_2} - p_{w_1}$ is the pressure jump across edge $e$, oriented according to $\mathbf{n}_1$.

For mass conservation, we write:
\[
\int_{\Omega^{\epsilon}} q_h \, \mathrm{div}(\mathbf{u}_h) = \sum_{w \in \mathcal{T}_h} q_w \int_{\partial w} \mathbf{u}_h \cdot \mathbf{n} = \sum_{e \in \mathcal{E}_h^0} |e| \, \llbracket q_h \rrbracket_e \, u_e.
\]
Substituting~\eqref{eq:veq_darcylaw} gives the pressure-only formulation:
\begin{equation}
a(p_h, q_h) := \sum_{e \in \mathcal{E}_h^0} \bar{\kappa}_e |e|^2 \, \llbracket p_h \rrbracket_e \, \llbracket q_h \rrbracket_e = \int_{\Omega^{\epsilon}} f q_h.
\label{eq:algbra_pressure}
\end{equation}

This formulation eliminates the velocity unknowns and reduces the mixed problem to a pressure-only system. The pressure $p_h$ is computed from~\eqref{eq:algbra_pressure}, and the velocity $\mathbf{u}_h$ is then recovered from~\eqref{eq:veq_darcylaw}.
This approach is equivalent to a cell-centered finite difference scheme on structured grids, as discussed in~\cite{arbogast1997mixed}.

Although the stiffness matrix in~\eqref{eq:algbra_pressure} is sparse, it remains computationally expensive to invert for large-scale problems. To address this, we introduce a multiscale framework in Sections~\ref{sec:offline_stage} and~\ref{sec:online_stage}, which significantly reduces the cost through localized basis functions.
In particular, we approximate the inverse by constructing multiscale operators $R_0$ and $R_0^T$, such that
\[
(B^T M^{-1} B)^{-1} \approx R_0 (R_0^T B^T M^{-1} B R_0)^{-1} R_0^T.
\]

\section{Offline basis functions} \label{sec:offline_stage}

In this section, we present the offline stage of our method, which constructs the multiscale approximation space for the pressure variable prior to the simulation. This stage comprises three steps. First, we generate local snapshot basis functions by incorporating fine-scale information. Second, we solve a local spectral problem to reduce the dimensionality of the snapshot space. The offline space is then formed by selecting eigenfunctions corresponding to the smallest eigenvalues. Finally, the global pressure solution is computed in this reduced offline space.

We define $\mathcal{T}_H$ as the coarse partition of $\Omega^{\epsilon}$. Each coarse element $K \in \mathcal{T}_H$ is a union of fine-grid elements from $\mathcal{T}_h$ (see Figure~\ref{fig:perforateddomain_mesh} for an illustration). Let $N_c$ denote the number of coarse elements.

\begin{figure}[htbp]
\centering
\begin{tikzpicture}
\centering
\node[anchor=south west,inner sep=0] at (0,0)
{\includegraphics[height=7cm]{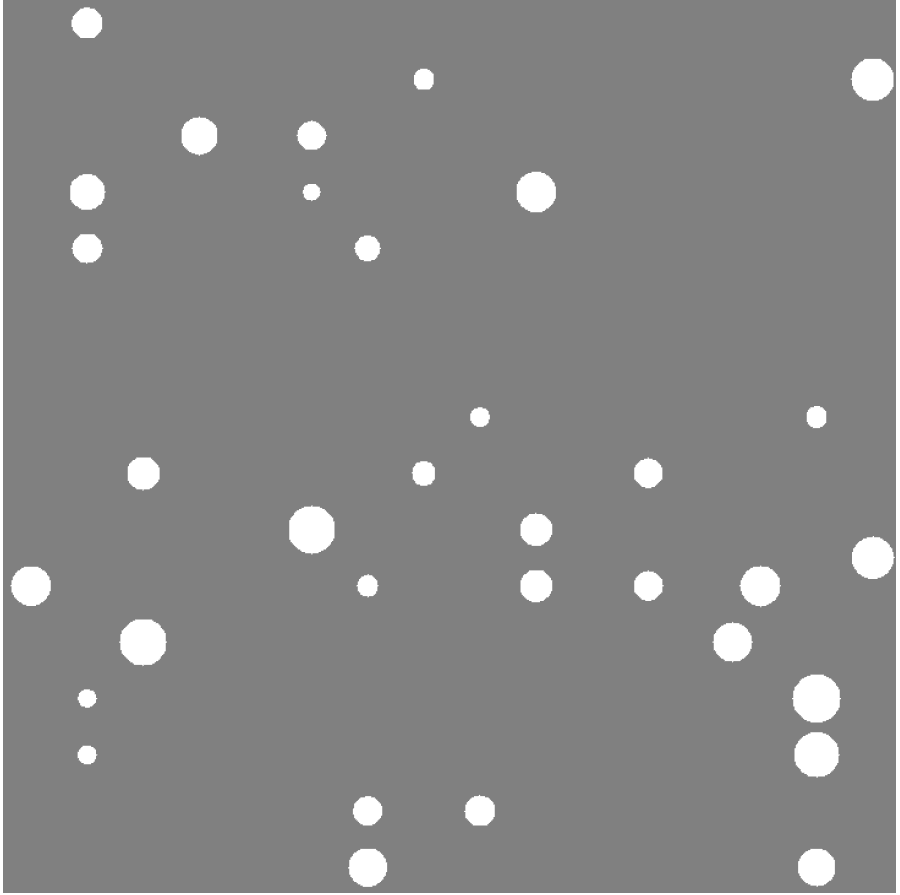}};
\draw[line width=1.5pt,step=0.7cm] (0,0) grid (7,7);
\draw[red,line width=2pt] (5.6,4.2) rectangle (6.3,4.9);
\draw[step=0.5cm,shift={(8.5,2)}] (0,0) grid (3,3);
\draw[line width=2pt,->] (6.4,4.5) -- (8,3.5);
\end{tikzpicture}
\caption{Illustration of the coarse grid and the fine-grid. 
For perforated regions, fine-grid cells that lie entirely inside the perforations are excluded from the computational domain, while the remaining cells adjacent to the perforation boundaries are treated with appropriate boundary conditions.}
\label{fig:perforateddomain_mesh}
\end{figure}
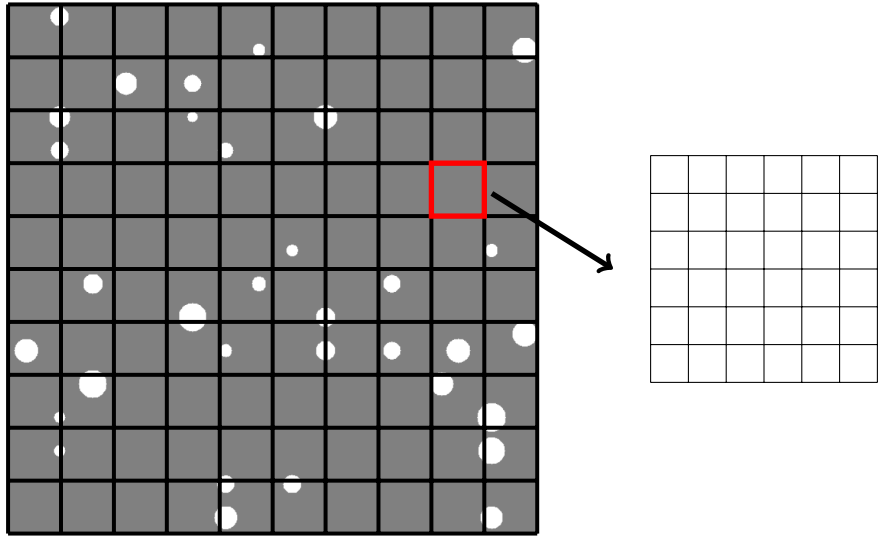

There are several strategies to construct the snapshot space, including using all fine-grid basis functions or solving localized problems with various boundary conditions. For simplicity, we adopt the approach that utilizes all fine-grid basis functions within each coarse block, i.e.,
$Q_{\textnormal{snap}}^i \coloneqq \mathrm{span}_{w \in K_i} \{ \varphi_w \}$.
Details of alternative constructions can be found in \cite{chen2020generalized}.

To reduce the dimension of $Q_{\textnormal{snap}}^i$, we solve the following spectral problem in each coarse block $K_i$:
\begin{equation}
a_i(\zeta_j^{i,\textnormal{off}}, q) = \lambda_j^i \, s_i(\zeta_j^{i,\textnormal{off}}, q), \quad \forall q \in Q^{\textnormal{snap}}_i,
\label{eq:spetral_problem_off}
\end{equation}
where the bilinear forms are defined as
\[
a_i(p, q) = \sum_{e \in \mathcal{E}_h^{0,i}} \bar{\kappa}_e |e|^2 \, \llbracket p \rrbracket_e \llbracket q \rrbracket_e, \quad
s_i(p, q) = \int_{K_i} \kappa \, p q.
\]
Here, $\mathcal{E}_h^{0,i}$ is the set of fine-grid interior edges within $K_i$. The eigenvalues are sorted in ascending order:
\[
0 = \lambda_1^i \le \lambda_2^i \le \cdots.
\]
The local offline space is spanned by the eigenfunctions corresponding to the first $l_i$ smallest eigenvalues:
\[
Q_{\textnormal{off}}^i = \mathrm{span}_j \{ \zeta_j^{i,\textnormal{off}} \}.
\]
Let $\Lambda = \min_{1 \leq i \leq N_c} \lambda_{l_i+1}^i$ denote the smallest eigenvalue not included in the offline space.

Unlike classical GMsFEM \cite{efendiev2013generalized}, our method does not rely on partition of unity functions to construct conforming basis functions. The global offline space is thus defined by a direct sum:
\[
Q_{\textnormal{ms}} = \bigoplus_i Q_{\textnormal{off}}^i = \mathrm{span}_{i,j} \{ \zeta_j^{i,\textnormal{off}} \}.
\]
The first eigenvalue equals zero, corresponding to a constant basis function, indicating that the pressure multiscale space enriches the coarse-grid $P_0$ space. Because it is a low-dimensional subspace of the fine-grid space, local mass conservation holds only at the coarse-grid level.

Let $R_0$ denote the matrix whose columns are the coefficients of the multiscale basis functions with respect to the fine-grid basis. Then,
\[
R_0 = \{ \zeta_1^{i,\textnormal{off}}, \zeta_2^{i,\textnormal{off}}, \ldots, \zeta_{l_i}^{i,\textnormal{off}} \}.
\]
Using the upscaling operator $R_0$ and its transpose $R_0^T$, the fine-grid operator $(B^T M^{-1} B)^{-1}$ is approximated by
\[
(B^T M^{-1} B)^{-1} \approx R_0 (R_0^T B^T M^{-1} B R_0)^{-1} R_0^T.
\]
This formulation reflects the core idea of multiscale methods: fine-scale operators are upscaled for efficient coarse-level computations (e.g., matrix inversion) and subsequently downscaled to yield approximate fine-scale operators with lower computational cost.

\begin{algorithm}[htbp]
\caption{Offline stage} 
\begin{algorithmic}[1]
\REQUIRE Permeability field $\kappa$, perforated domain $\Omega^{\epsilon}$.
\STATE Generate the coarse mesh $\mathcal{T}_H$ and fine mesh $\mathcal{T}_h$.
\FOR {$1 \le i \le N_c$}
\STATE Solve the local spectral problem~\eqref{eq:spetral_problem_off} in the snapshot space $Q_{\textnormal{snap}}^i$.
\STATE Sort the eigenvalues in ascending order, i.e., $\lambda_1 \le \lambda_2 \le \cdots$.
\STATE Select the first $l_i$ eigenfunctions to construct the local offline space $Q_{\textnormal{off}}^i$.
\ENDFOR
\STATE Assemble the global offline multiscale space $Q_{\textnormal{ms}}$.
\STATE Compute the multiscale pressure:
\[
p_{\textnormal{ms}} = R_0 (R_0^T B^T M^{-1} B R_0)^{-1} R_0^T F.
\]
\STATE Recover the multiscale velocity $\mathbf{u}_{\textnormal{ms}}$ using Equation~\eqref{eq:veq_darcylaw}.
\ENSURE Multiscale solution $(\mathbf{u}_{\textnormal{ms}}, p_{\textnormal{ms}})$.
\end{algorithmic}
\label{alg:mgms_offline}
\end{algorithm}

The steps of the offline stage are summarized in Algorithm~\ref{alg:mgms_offline}. An oversampling strategy may also be employed to mitigate the influence of artificial boundary conditions in local problems; see \cite{chen2020generalized} for details.

\section{Online basis functions} \label{sec:online_stage}

In this section, we describe the online stage of the multiscale method. During this stage, the solution is iteratively improved by enriching the multiscale space with additional basis functions constructed in selected coarse blocks where the residual is large. We present an adaptive enrichment strategy for efficiently constructing and adding these online basis functions.

A key component of the online stage is a residual-based indicator that guides the enrichment process. For a coarse block $K_i$, we define the local residual operator $r_i : Q(K_i) \rightarrow \mathbb{R}$ by
\begin{equation}
r_i(q) = \int_{K_i} f q - a_i(p_{\textnormal{ms}}, q).
\end{equation}

We also define a velocity-weighted $L^2$-norm on $Q_{\textnormal{snap}}^i$:
\begin{equation}
\|q_h\|_{V_i}^2 = \int_{K_i} \kappa^{-1} \mathbf{v} \cdot \mathbf{v} + \sum_{e \in \partial K_i} \bar{\kappa}_e |e|^2 \llbracket q_h \rrbracket_e^2 = \sum_{e \in \mathcal{E}_h^i} \bar{\kappa}_e |e|^2 \llbracket q_h \rrbracket_e^2,
\end{equation}
where $\mathbf{v}$ is the velocity field corresponding to $q$ via~\eqref{eq:veq_darcylaw}. 

The norm of the residual operator $r_i$ is then defined by
\begin{equation}
\delta_i \coloneqq \|r_i\|_{V_i} = \sup_{q_h \in Q_{\textnormal{snap}}^i} \frac{|r_i(q_h)|}{\|q_h\|_{V_i}}.
\label{eq:defi_delta_Ki}
\end{equation}
Indeed, the velocity-weighted $L^2$-norm of $r_i$ can be regarded as a discrete analogue of the $H^{-1}$ norm in the continuous setting, 
where $\|q_h\|_{V_i}$ corresponds to a discrete $H^1$ norm in the pressure space.

To construct the online basis function $\beta_i$ for a selected coarse block $K_i$, we solve the following local problem:
\begin{equation}
\begin{cases}
\begin{aligned}
\int_{K_i} \kappa^{-1} \boldsymbol{\gamma}_i \mathbf{v}_h
- \int_{K_i} \beta_i \mathrm{div}(\mathbf{v}_h) &= 0, 
&\forall \mathbf{v}_h \in \mathbf{V}_h(K_i), \\
\int_{K_i} \mathrm{div}(\boldsymbol{\gamma}_i) q_h &= 
\int_{K_i} (f - \mathrm{div}(\mathbf{u}_{\textnormal{ms}}^{n-1})) q_h, \quad &\forall q_h \in Q_h(K_i),
\end{aligned}
\end{cases}
\label{eq:online_zeta_beta_Ki}
\end{equation}
subject to homogeneous Neumann boundary conditions for $\boldsymbol{\gamma}_i$. Oversampling techniques can also be applied to improve the accuracy of $\beta_i$; see \cite{he2021adaptive}.

The local problem~\eqref{eq:online_zeta_beta_Ki} can be reformulated in variational form:
\begin{equation}
a_i(\beta_i, q_h) = r_i(q_h), \quad \forall q_h \in Q_h(K_i).
\label{eq:online_beta_Ki}
\end{equation}
It follows that $\beta_i$ is linearly independent of any function in the previous multiscale space $Q_{\textnormal{ms}}$, since $r_i(q) = 0$ for all $q \in Q_{\textnormal{ms}}$.

The full procedure for the online stage, including the adaptive enrichment algorithm, is given in Algorithm~\ref{alg:mgms_online}. Let $Q_{\textnormal{ms}}^0$ denote the initial multiscale space constructed in the offline stage, and let $(\mathbf{u}_{\textnormal{ms}}^0, p_{\textnormal{ms}}^0)$ denote the corresponding multiscale solution. The enrichment parameter $\theta$ controls the number of coarse blocks selected in each iteration: $\theta = 1$ corresponds to uniform enrichment, and $\theta = 0$ implies no online enrichment. The number of enrichment iterations is specified by \texttt{Iter}.

\begin{algorithm}[htbp]
\caption{Online stage}
\begin{algorithmic}[1]
\REQUIRE Initial multiscale space $Q_{\textnormal{ms}}^0$, initial solution $(\mathbf{u}_{\textnormal{ms}}^0, p_{\textnormal{ms}}^0)$, adaptivity parameter $\theta$, number of iterations \texttt{Iter}.
\FOR{$1 \le n \le \texttt{Iter}$}
    \FOR{$1 \le i \le N_c$}
        \STATE Compute local residual indicator $\delta_i$ using~\eqref{eq:defi_delta_Ki}.
    \ENDFOR
    \STATE Sort $\{\delta_i\}$ in descending order and find the smallest $k$ such that
    \begin{equation}
    \sum_{i=1}^{k} \delta_i^2 \ge \theta \sum_{i=1}^{N_c} \delta_i^2.
    \label{eq:theta_defi}
    \end{equation}
    \FOR{$1 \le i \le k$}
        \STATE Compute the online basis function $\beta_i$ by solving~\eqref{eq:online_zeta_beta_Ki}.
    \ENDFOR
    \STATE Enrich the multiscale space: $Q_{\textnormal{ms}}^n = Q_{\textnormal{ms}}^{n-1} \oplus \{\beta_i\}_{i=1}^k$.
    \STATE Update the multiscale operator $R_0$.
    \STATE Solve for pressure: 
    \[
    p_{\textnormal{ms}}^n = R_0 (R_0^T B^T M^{-1} B R_0)^{-1} R_0^T F.
    \]
\ENDFOR
\STATE Recover the velocity $\mathbf{u}_{\textnormal{ms}}^n$ using~\eqref{eq:veq_darcylaw}.
\ENSURE Multiscale solution $(\mathbf{u}_{\textnormal{ms}}^n, p_{\textnormal{ms}}^n)$.
\end{algorithmic}
\label{alg:mgms_online}
\end{algorithm}

\section{Analysis} \label{sec:convergence_results}

In this section, we provide a theoretical analysis of the convergence behavior for both the offline and online stages of the proposed multiscale method. Specifically, we show that the accuracy of the offline solution is influenced by the smallest eigenvalue $\Lambda$, corresponding to the first excluded eigenfunction from the multiscale space. For the online stage, we demonstrate that the convergence rate depends on $\Lambda$, the adaptive parameter $\theta$, and the number of enrichment iterations \texttt{Iter}.

We begin by introducing the weak formulation used in our multiscale approach. The key idea is to construct a reduced-order subspace $Q_{\textnormal{ms}} \subset Q_h$ and solve Equation~\eqref{eq:algbra_pressure} within this space. The multiscale version of the mixed formulation~\eqref{eq:pde_mixed} is given as:

\begin{equation}
\begin{cases}
\begin{aligned}
\int_{\Omega^{\epsilon}} \kappa^{-1}\mathbf{u}_{\textnormal{ms}} \cdot \mathbf{v}_{\textnormal{ms}}
- \int_{\Omega^{\epsilon}} \mathrm{div}(\mathbf{v}_{\textnormal{ms}}) p_{\textnormal{ms}} 
&= 0, 
& \forall \mathbf{v}_{\textnormal{ms}} \in \mathbf{V}_{\textnormal{ms}}, \\
\int_{\Omega^{\epsilon}} \mathrm{div}(\mathbf{v}_{\textnormal{ms}}) q_{\textnormal{ms}}
&= \int_{\Omega^{\epsilon}} f q_{\textnormal{ms}}, 
& \forall q_{\textnormal{ms}} \in Q_{\textnormal{ms}}.
\end{aligned}
\end{cases}
\label{eq:weak_rt_ms}
\end{equation}
Here, the multiscale velocity space $\mathbf{V}_{\textnormal{ms}}$ coincides with the fine-grid velocity space $\mathbf{V}_h$. We retain the subscript “ms” for clarity and consistency. Using the definition of $a(\cdot,\cdot)$ in \eqref{eq:algbra_pressure} and the velocity-pressure relationship in \eqref{eq:veq_darcylaw}, we have

\begin{equation}
a(p_h, q_h) =
\sum_{e \in \mathcal{E}_h^0} \bar{\kappa}_e |e|^2 \llbracket p_h \rrbracket_e \llbracket q_h \rrbracket_e
= \int_{\Omega^{\epsilon}} \kappa^{-1} \mathbf{u}_h \cdot \mathbf{v}_h.
\label{eq:kpq_uv}
\end{equation}
This relation indicates that the error in the pressure solution can be evaluated through the corresponding velocity error.

\subsection{Error analysis of the offline stage}
In this subsection, we analyze the convergence of the offline multiscale solution. 
The error estimate depends on the smallest eigenvalue $\Lambda$ whose associated eigenfunction is not included in the offline space, which reflects the spectral-decay principle commonly used in the analysis of GMsFEM. 
Our proof structure follows the spectral analysis framework developed in~\cite{chung2015mixed, chen2020generalized}. 
In particular, we employ a pressure-space spectral decomposition and develop the analysis for the mixed formulation of the Darcy problem in perforated domains.

Let $\bar{f}$ denote the $L^2$-projection of the source term $f$ onto the pressure multiscale space $Q_{\textnormal{ms}}$. Then, for any $q_{\textnormal{ms}} \in Q_{\textnormal{ms}}$, it holds that
\[
\int_{K_i} (f - \bar{f}) q_{\textnormal{ms}} = 0.
\]

We now define an auxiliary local solution $\left(\widehat{\mathbf{u}}, \widehat{p}\right)$ satisfying
\begin{equation}
\begin{cases}
\begin{aligned}
\kappa^{-1} \widehat{\mathbf{u}} + \nabla \widehat{p} &= 0, & \text{in } K_i, \\
\mathrm{div}(\widehat{\mathbf{u}}) &= \bar{f}, & \text{in } K_i,
\end{aligned}
\end{cases}
\label{eq:pde_mixed_K_L2f}
\end{equation}
with the following compatibility conditions:
\begin{equation}
\widehat{\mathbf{u}} \cdot \mathbf{n} = \mathbf{u}_h \cdot \mathbf{n} \quad \text{on } \partial K_i,
\qquad
\int_{K_i} \widehat{p} = \int_{K_i} p_h.
\label{eq:constrain_uphat}
\end{equation}
The solution $(\widehat{\mathbf{u}}, \widehat{p})$ is approximated in the finite element spaces $V_h(K_i) \times Q_h(K_i)$.

We next establish a bound on the discrepancy between the fine-grid solution and this auxiliary solution.

\begin{lemma} \label{lemma:uhu_kdiv}
Let $\mathbf{u}_h \in \mathbf{V}_h$ be the fine-grid solution of~\eqref{eq:weak_rt_fine}, and let $\widehat{\mathbf{u}} \in \mathbf{V}_h$ be the solution of~\eqref{eq:pde_mixed_K_L2f}. Then, we have
\begin{equation}
\int_{\Omega^{\epsilon}} \kappa^{-1} |\mathbf{u}_h - \widehat{\mathbf{u}}|^2 
\leq \kappa_{\min}^{-1} \int_{\Omega^{\epsilon}} |f - \bar{f}|^2.
\end{equation}
where $\kappa_{\min} = \min_{x \in \Omega^{\epsilon}} \kappa$.
\end{lemma}

\begin{proof}
Let $\mathbf{u}_h$ and $\widehat{\mathbf{u}}$ denote the solutions to problems~\eqref{eq:weak_rt_fine} and~\eqref{eq:pde_mixed_K_L2f}, respectively. Then, for each coarse block $K_i$, the following relations hold:
\begin{equation}
\begin{aligned}
\int_{K_i} \kappa^{-1}(\mathbf{u}_h - \widehat{\mathbf{u}}) \cdot \mathbf{v}_h
- \int_{K_i} (p_h - \widehat{p}) \, \mathrm{div}(\mathbf{v}_h) &= 0, 
&\quad \forall \mathbf{v}_h \in \mathbf{V}_h(K_i), \\
\int_{K_i} \mathrm{div}(\mathbf{u}_h - \widehat{\mathbf{u}}) \, q_h &= \int_{K_i} (f - \bar{f}) \, q_h, 
&\quad \forall q_h \in Q_h(K_i).
\end{aligned}
\label{eq:weak_diff_uhuL2_K}
\end{equation}
Choosing $\mathbf{v}_h = \mathbf{u}_h - \widehat{\mathbf{u}}$ and $q_h = p_h - \widehat{p}$ in~\eqref{eq:weak_diff_uhuL2_K}, we obtain
\begin{equation}
\int_{K_i} \kappa^{-1} |\mathbf{u}_h - \widehat{\mathbf{u}}|^2 = 
\int_{K_i} (f - \bar{f}) (p_h - \widehat{p}).
\label{eq:diff_uhu_ffL2_php}
\end{equation}
From the standard inf-sup condition for mixed elements~\cite{boffi2013mixed}, we have
\begin{equation}
\|q_h\|_{L^2(K_i)} \leq 
\sup_{\mathbf{v}_h \in \mathbf{V}_h(K_i)} 
\frac{\int_{K_i} q_h \, \mathrm{div}(\mathbf{v}_h)}{\|\mathbf{v}_h\|_{H(\mathrm{div};K_i)}}.
\label{eq:infsup_ineq}
\end{equation}
Applying~\eqref{eq:infsup_ineq} with $q_h = p_h - \widehat{p}$ and using the first equation in~\eqref{eq:weak_diff_uhuL2_K}, we derive
\begin{equation}
\|p_h - \widehat{p}\|_{L^2(K_i)}^2 \leq 
\kappa_{\min}^{-1} 
\int_{K_i} \kappa^{-1} |\mathbf{u}_h - \widehat{\mathbf{u}}|^2.
\label{eq:phphat_uhuhat}
\end{equation}
Substituting~\eqref{eq:phphat_uhuhat} into~\eqref{eq:diff_uhu_ffL2_php} and using the Cauchy–Schwarz inequality, we obtain
\begin{equation}
\int_{K_i} \kappa^{-1} |\mathbf{u}_h - \widehat{\mathbf{u}}|^2 \leq
\kappa_{\min}^{-1} \int_{K_i} |f - \bar{f}|^2.
\label{eq:diff_uhuhat_ffL2_Ki}
\end{equation}
Summing~\eqref{eq:diff_uhuhat_ffL2_Ki} over all coarse elements \(K_i \subset \Omega^{\epsilon}\) yields the desired estimate.
\end{proof}

Based on the above lemma, we now establish the convergence of the offline multiscale solution.

\begin{theorem} \label{thm:uhms_offline}
Let $\mathbf{u}_h$ be the fine-grid solution of~\eqref{eq:weak_rt_fine}, and $\mathbf{u}_{\textnormal{ms}}$ be the multiscale solution from~\eqref{eq:weak_rt_ms}. Then,
\begin{equation}
\int_{\Omega^{\epsilon}} \kappa^{-1} |\mathbf{u}_h - \mathbf{u}_{\textnormal{ms}}|^2 
\le \frac{\kappa_{\min}^{-1}}{\Lambda} \sum_{i=1}^{N_c} a_i(\widehat{p}, \widehat{p}) 
+ \kappa_{\min}^{-1} \sum_{i=1}^{N_c} \int_{K_i} |f - \bar{f}|^2,
\label{eq:uhms_phat_ffbar}
\end{equation}
where $\widehat{p}$ is the auxiliary pressure defined in~\eqref{eq:pde_mixed_K_L2f}.
\end{theorem}

\begin{proof}
Let $\widehat{\mathbf{u}}$ and $\mathbf{u}_{\textnormal{ms}}$ be the solutions of problems~\eqref{eq:pde_mixed_K_L2f} and~\eqref{eq:weak_rt_ms}, respectively. We have

\begin{equation}
\begin{aligned}
\int_{\Omega^{\epsilon}} \kappa^{-1}(\widehat{\mathbf{u}} - \mathbf{u}_{\textnormal{ms}}) \cdot \mathbf{v}_{\textnormal{ms}} 
- \int_{\Omega^{\epsilon}} (\widehat{p}-p_{\textnormal{ms}}) \, \mathrm{div}(\mathbf{v}_{\textnormal{ms}}) &= 0, 
& \forall \mathbf{v}_{\textnormal{ms}} \in \mathbf{V}_{\textnormal{ms}}, \\
\int_{\Omega^{\epsilon}} q_{\textnormal{ms}} \, \mathrm{div}(\widehat{\mathbf{u}} - \mathbf{u}_{\textnormal{ms}}) &= 0, 
& \forall q_{\textnormal{ms}} \in Q_{\textnormal{ms}},
\end{aligned}
\label{eq:weak_diff_uhatums}
\end{equation}
where the second equation follows from the definition of $\bar{f}$.

Since $\mathbf{V}_{\textnormal{ms}} = \mathbf{V}_h$, we can take $\mathbf{v}_{\textnormal{ms}} = \mathbf{u}_h - \mathbf{u}_{\textnormal{ms}}$ in the first equation of~\eqref{eq:weak_diff_uhatums} to obtain
\begin{equation}
\int_{\Omega^{\epsilon}} \kappa^{-1} (\widehat{\mathbf{u}} - \mathbf{u}_{\textnormal{ms}}) \cdot (\mathbf{u}_h - \mathbf{u}_{\textnormal{ms}})
= \int_{\Omega^{\epsilon}} (\widehat{p} - p_{\textnormal{ms}}) \mathrm{div}(\mathbf{u}_h - \mathbf{u}_{\textnormal{ms}}) 
= \int_{\Omega^{\epsilon}} (\widehat{p} - \widehat{p}_{\textnormal{ms}}) \mathrm{div}(\mathbf{u}_h - \mathbf{u}_{\textnormal{ms}}),
\label{eq:hatphpms_uhums}
\end{equation}
where $\widehat{p}_{\textnormal{ms}}$ is the $Q_{\textnormal{ms}}$-projection of $\widehat{p}$.

Let $\{ \zeta_1^{i,\textnormal{off}}, \zeta_2^{i,\textnormal{off}}, \cdots \}$ be the eigenfunctions of~\eqref{eq:spetral_problem_off} on $K_i$, and suppose the first $l_i$ are selected in the offline space. We write
\[
\widehat{p} = \sum_{j} c_j^i \zeta_j^{i,\textnormal{off}}, \quad
\widehat{p}_{\textnormal{ms}} = \sum_{j=1}^{l_i} c_j^i \zeta_j^{i,\textnormal{off}}.
\]
Using~\eqref{eq:veq_darcylaw}, we have
\begin{equation}
\int_{K_i} |\widehat{p} - \widehat{p}_{\textnormal{ms}}|^2
\le \kappa_{\textnormal{min}}^{-1} s_i(\widehat{p} - \widehat{p}_{\textnormal{ms}}, \widehat{p} - \widehat{p}_{\textnormal{ms}})
\le \frac{1}{\lambda_{l_i+1}^i} \, a_i(\widehat{p}, \widehat{p}).
\label{eq:uhatmshms_phat}
\end{equation}
The proof is completed by combining~\eqref{eq:hatphpms_uhums}, ~\eqref{eq:uhatmshms_phat}, and Lemma~\ref{lemma:uhu_kdiv}.
\end{proof}

\paragraph{Remark.} 
The error estimate in Theorem~\ref{thm:uhms_offline} consists of two contributions: 
the spectral error from excluding higher eigenmodes (involving $\Lambda$), and 
the projection error of the source term $f$ onto the multiscale space. 
The first term decreases as more eigenfunctions are included, while the second vanishes when $f$ is well-represented in the offline space.
Together, they characterize the approximation power of the offline multiscale space.

\subsection{Error analysis of the online stage}

We now analyze the convergence behavior of the online stage. First, we show that the error of each multiscale approximation can be bounded in terms of the residual indicators $\delta_i$ and the spectral parameter $\Lambda$. We then prove a decay estimate for the online iterations based on the adaptive parameter $\theta$ and the number of enrichment steps.

\begin{lemma} \label{lem:u_hms_ri}
Let $(\mathbf{u}_h, p_h)$ be the fine-grid solution of~\eqref{eq:weak_rt_fine}, and $(\mathbf{u}_{\textnormal{ms}}, p_{\textnormal{ms}})$ be the multiscale solution of~\eqref{eq:weak_rt_ms}. Then,
\begin{equation}
\int_{\Omega^{\epsilon}} \kappa^{-1} 
|\mathbf{u}_h - \mathbf{u}_{\textnormal{ms}}|^2 
\le \frac{2}{\Lambda} \sum_{i=1}^{N_c} \delta_i^2.
\end{equation}
\end{lemma}

\begin{proof}
Let $q_h$ be an arbitrary function in $Q_h$, and let $\mathbf{v}_h$ denote the corresponding velocity defined via~\eqref{eq:veq_darcylaw}. 
Using~\eqref{eq:kpq_uv}, we obtain
\begin{equation}
\begin{aligned}
\int_{\Omega^{\epsilon}} \kappa^{-1} (\mathbf{u}_h - \mathbf{u}_{\textnormal{ms}}) \cdot \mathbf{v}_h 
&= a(p_h - p_{\textnormal{ms}}, q_h) 
= \int_{\Omega^{\epsilon}} f q_h - a(p_{\textnormal{ms}}, q_h) \\
&= \int_{\Omega^{\epsilon}} f (q_h - \widehat{q}_h) 
- a(p_{\textnormal{ms}}, q_h - \widehat{q}_h) \\
&= \sum_{i=1}^{N_c} \int_{K_i} (q_h - \widehat{q}_h) \big(f - \mathrm{div}(\mathbf{u}_{\textnormal{ms}})\big),
\end{aligned}
\label{eq:u_hms_vh_qhoff}
\end{equation}
where $\widehat{q}_h$ is the projection of $q_h$ onto the multiscale pressure space $Q_{\textnormal{ms}}$, and the penultimate equality follows from the Galerkin orthogonality: $\int_{\Omega^{\epsilon}} f \widehat{q}_h = a(p_{\textnormal{ms}}, \widehat{q}_h)$.

By the definition of the local residual operator $r_i$, we obtain
\[
\int_{\Omega^{\epsilon}} \kappa^{-1} (\mathbf{u}_h - \mathbf{u}_{\textnormal{ms}}) \cdot \mathbf{v}_h 
= \sum_{i=1}^{N_c} r_i(q_h - \widehat{q}_h) 
\le \sum_{i=1}^{N_c} \Vert r_i \Vert_{V_i} \, \Vert q_h - \widehat{q}_h \Vert_{V_i}.
\]
To estimate $\Vert q_h - \widehat{q}_h \Vert_{V_i}$, we recall the definition of the local norm:
\begin{equation}
\begin{aligned}
\Vert q_h - \widehat{q}_h \Vert_{V_i}^2 &= 
\int_{K_i} \kappa^{-1} |\mathbf{v}_h - \widehat{\mathbf{v}}_h|^2
+ \sum_{e \in \partial K_i} \bar{\kappa}_e |e|^2
\llbracket q_h - \widehat{q}_h \rrbracket_e^2 \\
&= \sum_{e \in \mathcal{E}_h^0(K_i)} \bar{\kappa}_e |e|^2
\llbracket q_h - \widehat{q}_h \rrbracket_e^2 +
\sum_{e \in \partial K_i} \bar{\kappa}_e |e|^2
\llbracket q_h - \widehat{q}_h \rrbracket_e^2 \\
&\le 2 \sum_{w \in K_i} \kappa_w |q_h - \widehat{q}_h|^2,
\end{aligned}
\label{eq:q_hoff_Vi}
\end{equation}
where $\widehat{\mathbf{v}}_h$ is defined via~\eqref{eq:veq_darcylaw} using $\widehat{p}_h$.

Let $\{ \zeta_1^{i,\textnormal{off}}, \zeta_2^{i,\textnormal{off}}, \ldots \}$ be the eigenfunctions of~\eqref{eq:spetral_problem_off} on $K_i$, and assume that the first $l_i$ eigenfunctions are used in the offline space.  
Suppose $q_h = \sum_j c_j^i \zeta_j^{i,\textnormal{off}}$, and its projection $\widehat{q}_h = \sum_{j=1}^{l_i} c_j^i \zeta_j^{i,\textnormal{off}}$. Then,
\begin{equation}
\begin{aligned}
s_i(q_h - \widehat{q}_h, q_h - \widehat{q}_h) 
&= \sum_{j \ge l_i+1} (c_j^i)^2
\le \frac{1}{\lambda_{l_i+1}^i} \sum_{j \ge l_i+1} \lambda_j^i (c_j^i)^2 \\
&\le \frac{1}{\lambda_{l_i+1}^i} a_i(q_h, q_h)
\le \frac{1}{\lambda_{l_i+1}^i} \int_{K_i} \kappa^{-1} \mathbf{v}_h^2.
\end{aligned}
\label{eq:si_qhoff_vh}
\end{equation}
Finally, combining~\eqref{eq:u_hms_vh_qhoff},~\eqref{eq:q_hoff_Vi}, and~\eqref{eq:si_qhoff_vh}, and taking $\mathbf{v}_h = \mathbf{u}_h - \mathbf{u}_{\textnormal{ms}}$, we complete the proof.
\end{proof}

Now, we can present the convergence analysis of the online adaptive enrichment algorithm.

\begin{theorem} \label{thm:uhms_online}
Let $\mathbf{u}_h$ be the fine-grid solution of~\eqref{eq:weak_rt_fine}, and $\mathbf{u}_{\textnormal{ms}}^n$ be the multiscale solution after $n$ online enrichment iterations. Then,
\begin{equation}
\int_{\Omega^{\epsilon}} \kappa^{-1} |\mathbf{u}_h - \mathbf{u}_{\textnormal{ms}}^n|^2 
\le \left(1 - \frac{\theta \Lambda}{2}\right)^n \int_{\Omega^{\epsilon}} \kappa^{-1} |\mathbf{u}_h - \mathbf{u}_{\textnormal{ms}}^0|^2.
\label{ineq:uhumsn_online}
\end{equation}
\end{theorem}

\begin{proof}
We begin by establishing the convergence result for a single enrichment step, from which the full estimate in~\eqref{ineq:uhumsn_online} can be obtained by induction.

Let $(\mathbf{u}_{\textnormal{ms}}^{n-1}, p_{\textnormal{ms}}^{n-1})$ be the multiscale solution at the $(n-1)$-th iteration, and let $Q_{\textnormal{ms}}^{n-1}$ denote the corresponding pressure multiscale space. Assume that we add one online basis function $\beta_i$ constructed by~\eqref{eq:online_zeta_beta_Ki}, and define the enriched space by $Q_{\textnormal{ms}}^n = Q_{\textnormal{ms}}^{n-1} \oplus \{\beta_i\}$.
The new multiscale solution $(\mathbf{u}_{\textnormal{ms}}^n, p_{\textnormal{ms}}^n)$ satisfies
\[
\int_{\Omega^{\epsilon}} \kappa^{-1}
|\mathbf{u}_h - \mathbf{u}_{\textnormal{ms}}^n|^2
= a(p_h - p_{\textnormal{ms}}^n, p_h - p_{\textnormal{ms}}^n)
= \inf_{q \in Q_{\textnormal{ms}}^n} a(p_h - q, p_h - q).
\]
Taking a test function of the form $q = p_{\textnormal{ms}}^{n-1} + c \beta_i$, we obtain
\begin{align*}
\int_{\Omega^{\epsilon}} \kappa^{-1} |\mathbf{u}_h - \mathbf{u}_{\textnormal{ms}}^n|^2
&\le a(p_h - p_{\textnormal{ms}}^{n-1} - c \beta_i, p_h - p_{\textnormal{ms}}^{n-1} - c \beta_i) \\
&= a(p_h - p_{\textnormal{ms}}^{n-1}, p_h - p_{\textnormal{ms}}^{n-1}) 
- 2c\, a(p_h - p_{\textnormal{ms}}^{n-1}, \beta_i) + c^2 a(\beta_i, \beta_i).
\end{align*}
The optimal choice of $c$ that minimizes the right-hand side is given by $c = a(p_h - p_{\textnormal{ms}}^{n-1}, \beta_i) / a(\beta_i, \beta_i)$. Substituting this into the expression yields
\begin{align*}
\int_{\Omega^{\epsilon}} \kappa^{-1} |\mathbf{u}_h - \mathbf{u}_{\textnormal{ms}}^n|^2
&\le \int_{\Omega^{\epsilon}} \kappa^{-1} |\mathbf{u}_h - \mathbf{u}_{\textnormal{ms}}^{n-1}|^2 
- \frac{[a(p_h - p_{\textnormal{ms}}^{n-1}, \beta_i)]^2}{a(\beta_i, \beta_i)} \\
&= \int_{\Omega^{\epsilon}} \kappa^{-1} |\mathbf{u}_h - \mathbf{u}_{\textnormal{ms}}^{n-1}|^2 
- \frac{[(f, \beta_i) - a(p_{\textnormal{ms}}^{n-1}, \beta_i)]^2}{a(\beta_i, \beta_i)} \\
&= \int_{\Omega^{\epsilon}} \kappa^{-1} |\mathbf{u}_h - \mathbf{u}_{\textnormal{ms}}^{n-1}|^2 
- \| r_i \|_{V_i}^2.
\end{align*}

Applying Lemma~\ref{lem:u_hms_ri}, we deduce that
\[
\int_{\Omega^{\epsilon}} \kappa^{-1} |\mathbf{u}_h - \mathbf{u}_{\textnormal{ms}}^n|^2 
\le \left( 1 - \frac{\Lambda \| r_i \|_{V_i}^2}{2 \sum_{j=1}^{N_c} \| r_j \|_{V_j}^2} \right)
\int_{\Omega^{\epsilon}} \kappa^{-1} |\mathbf{u}_h - \mathbf{u}_{\textnormal{ms}}^{n-1}|^2.
\]
When $k$ online basis functions are added, selected based on the residual indicators $\delta_i$, the estimate generalizes to
\[
\int_{\Omega^{\epsilon}} \kappa^{-1} |\mathbf{u}_h - \mathbf{u}_{\textnormal{ms}}^n|^2 
\le \left(1 - \frac{\theta \Lambda}{2}\right) 
\int_{\Omega^{\epsilon}} \kappa^{-1} |\mathbf{u}_h - \mathbf{u}_{\textnormal{ms}}^{n-1}|^2.
\]
Using this inequality inductively over the number of enrichment steps yields the desired result.
\end{proof}

\paragraph{Remark.} 
The convergence behavior of the online enrichment algorithm is governed by the minimal excluded eigenvalue $\Lambda$, the enrichment threshold $\theta$, and the number of enrichment steps $n$. 
The case $\theta = 1$ corresponds to uniform enrichment and leads to the most rapid error decay. 
When $\Lambda$ is large, the offline space is already sufficiently rich, and only a few online updates are necessary.

\section{Numerical results} \label{sec:numerical_results}
In this section, we evaluate the performance of our approach using three different models. Figure \ref{fig:domains} provides an illustration of these model domains. Models 1 and 2 are synthetic domains constructed by removing randomly sized circles (which may overlap) from an initially rectangular domain. Model 3 consists of a two-dimensional slice of a sandstone sample, as described in \cite{mehmani2021multiscale,berg2018gildehauser}. 
All permeability fields and mesh grids were generated using the MATLAB Reservoir Simulation Toolbox (MRST) \cite{lie2019introduction}.
The fine grid comprises 511,756 cells for Model 1, 264,863 cells for Model 2, and 338,927 cells for Model 3. 
We consider two different coarse mesh sizes in this study to examine their effect on the solution. The first coarse grid uses a spacing of $H$ such that each coarse block contains 20 fine grid cells in each coordinate direction. The second coarse grid uses a spacing of $H/2$, i.e., a refinement by a factor of two in each direction. We impose Dirichlet boundary conditions of $p=1$ on the left boundary and $p=0$ on the right boundary. This setup corresponds to a unit pressure drop across the domain. All other boundaries are assigned no-flux conditions (i.e., $\mathbf{u} \cdot \mathbf{n} = 0$). The relative $L^2$ errors in pressure and velocity are defined as follows:

\[
e_p = \frac{\int_{\Omega^{\epsilon}} |p_h - p_{\textnormal{ms}}|^2}{\int_{\Omega^{\epsilon}} |p_h|^2}, \quad
e_{\mathbf{u}} = \frac{\int_{\Omega^{\epsilon}} |\mathbf{u}_h - \mathbf{u}_{\textnormal{ms}}|^2}{\int_{\Omega^{\epsilon}} |\mathbf{u}_h|^2}.
\]

Here $p_h$ and $\mathbf{u}_h$ denote the fine-scale pressure and velocity solutions, and $p_{\textnormal{ms}}$ and $\mathbf{u}_{\textnormal{ms}}$ are the corresponding multiscale solutions. In our computations, we employ several additional techniques to further reduce the dimensionality of the multiscale solution space in each local region. In the offline stage, for example, we exclude any of the first $l_i$ eigenfunctions from the offline space if their corresponding eigenvalues exceed $100{,}000$. Our analysis indicates that convergence is primarily governed by the eigenfunctions associated with the smallest eigenvalues; hence, those basis functions corresponding to very large eigenvalues (above this threshold) can be safely omitted from the offline multiscale space. We denote by $L$ the prescribed number of offline basis functions per coarse block. Consequently, due to this eigenfunction exclusion and the constraints imposed by the local geometry, the actual number of basis functions per coarse block may be less than $L$. In the online stage, we adopt a stopping criterion for local enrichment: if the local residual falls below $0.1\%$, no new online basis function is constructed. This threshold is applied in both the uniform and adaptive enrichment strategies. Such a criterion avoids constructing unnecessary basis functions in regions where the solution is already sufficiently accurate.

\begin{figure}[htbp]
\centering
\begin{subfigure}[b]{0.32\textwidth}
\includegraphics[width=\linewidth]{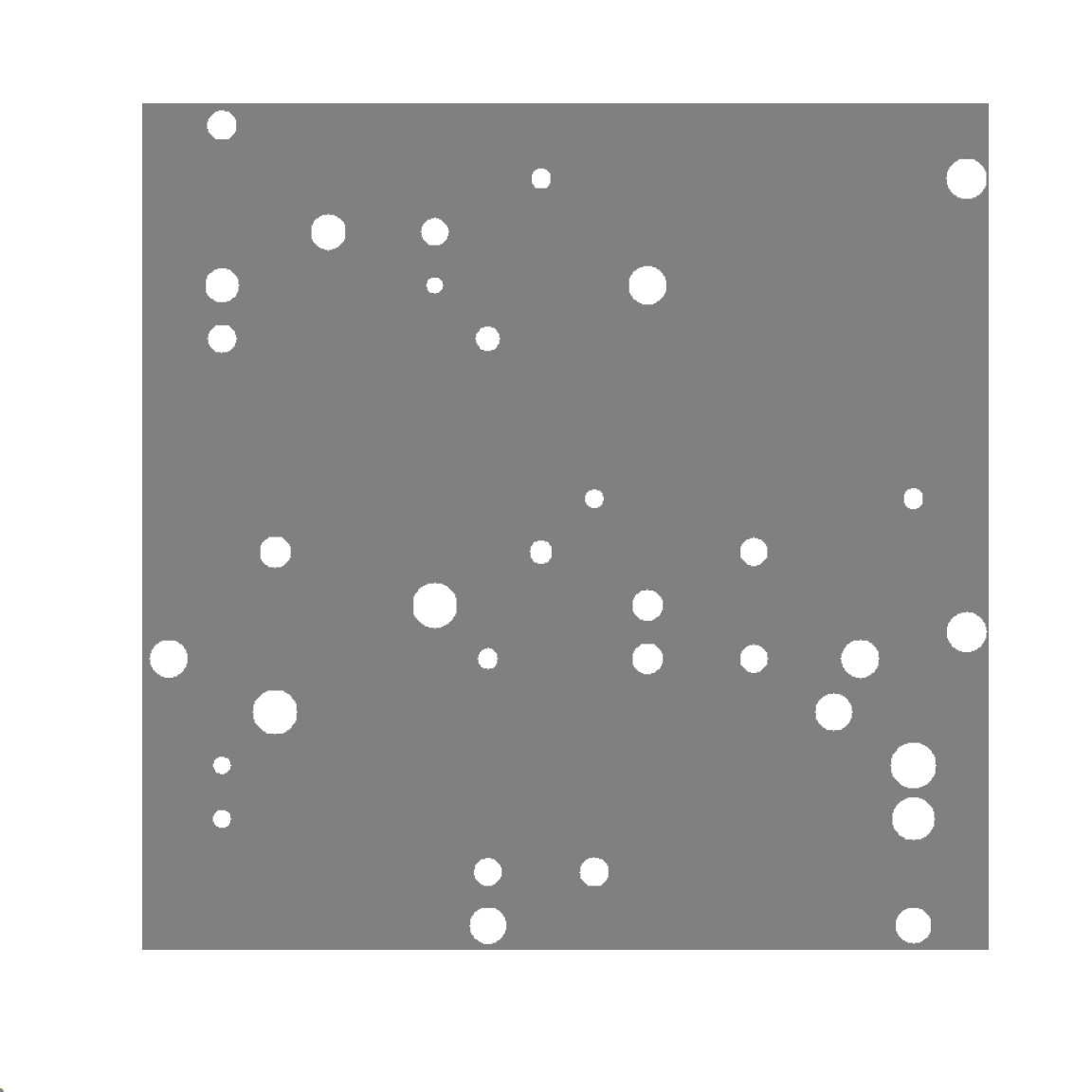}
\caption{Model 1}
\end{subfigure}
\begin{subfigure}[b]{0.32\textwidth}
\includegraphics[width=\linewidth]{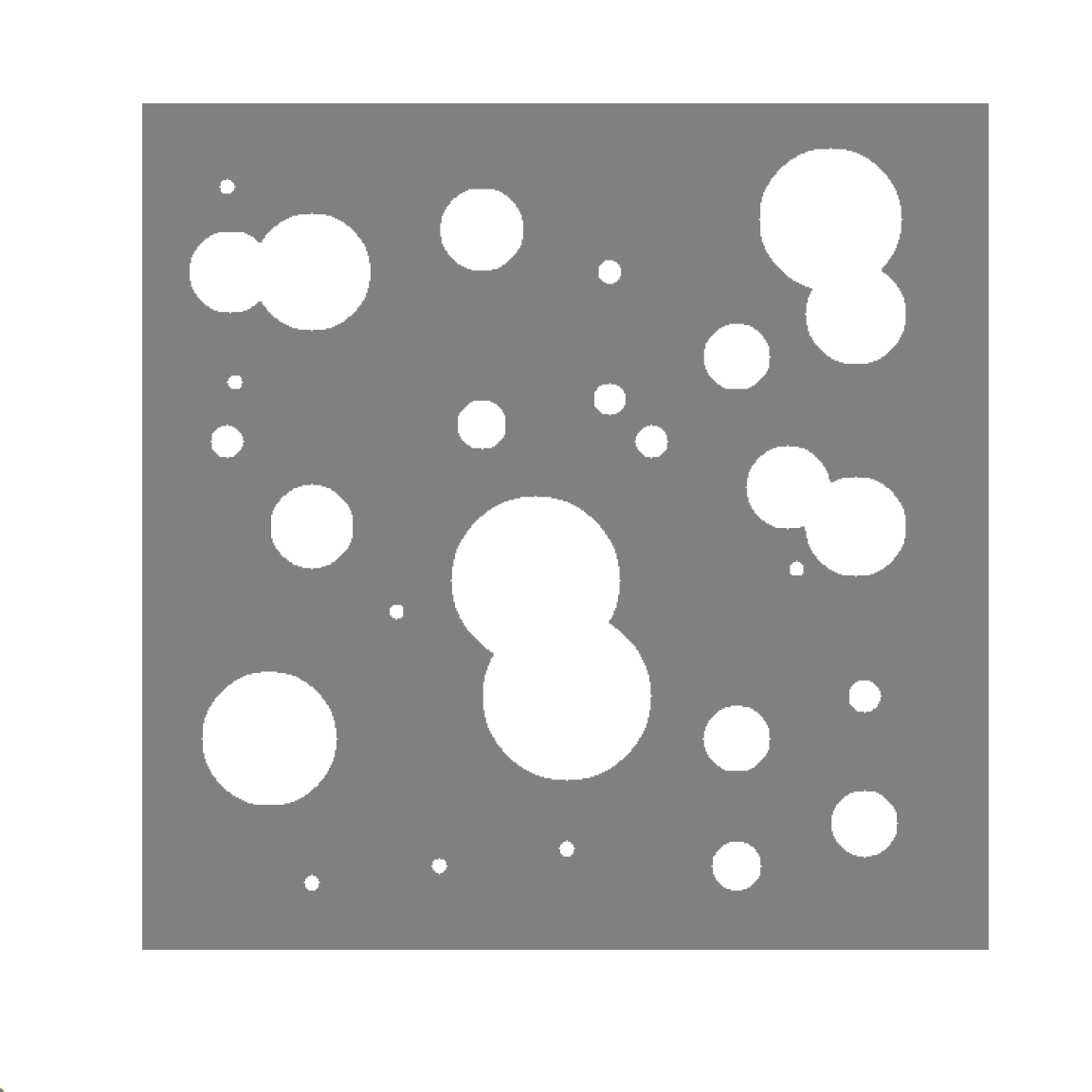}
\caption{Model 2}
\end{subfigure}
\begin{subfigure}[b]{0.32\textwidth}
\includegraphics[width=\linewidth]{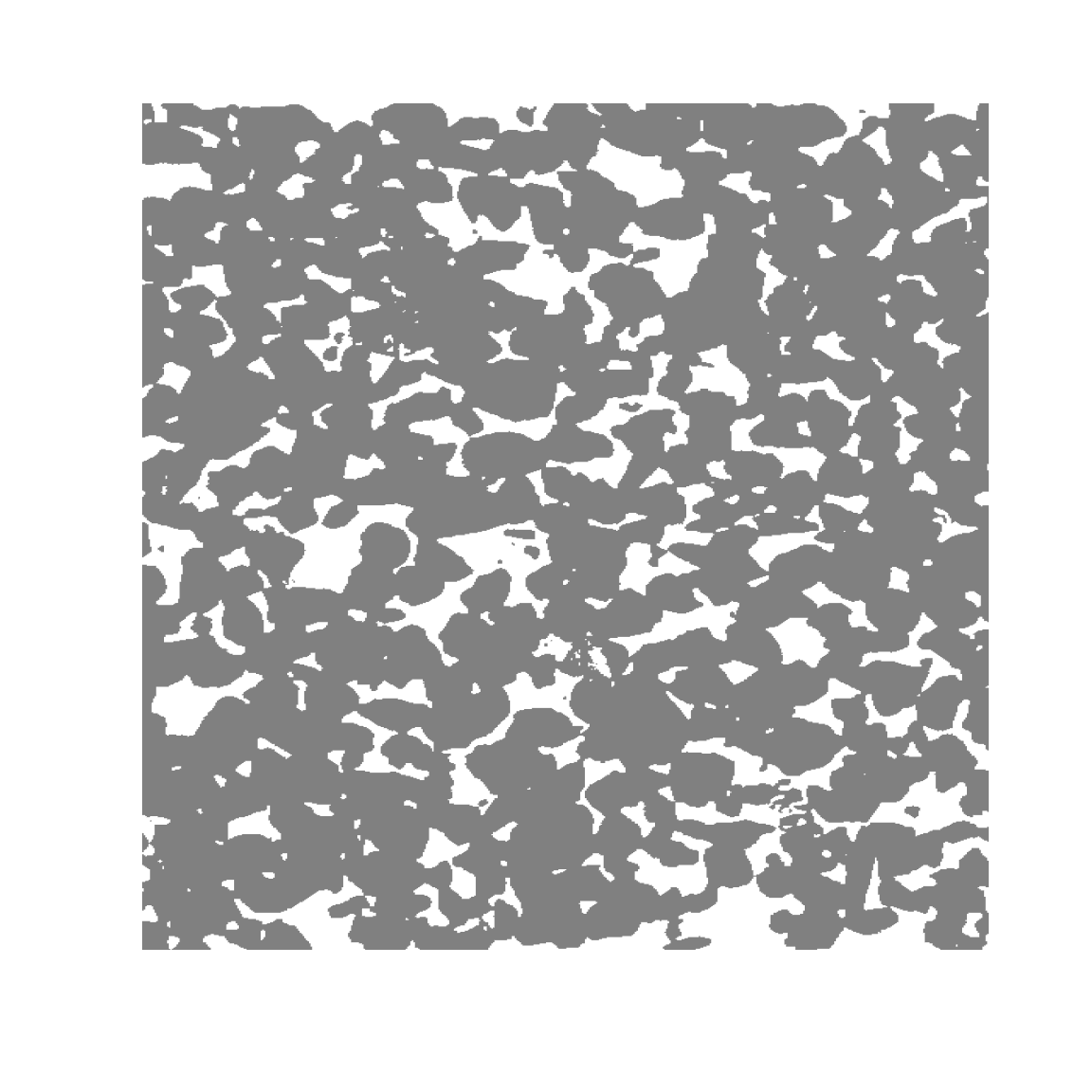}
\caption{Model 3}
\end{subfigure}
\caption{Three different perforated media. The fine grid consists of 511,756 cells for Model 1; 264,863 for Model 2; and 338,927 for Model 3.}
\label{fig:domains}
\end{figure}

\subsection{Effect of offline parameters}
In this subsection, we present offline simulation results obtained using varying numbers of offline basis functions and different coarse partitions. Detailed results for the three test cases are provided in Tables \ref{tab:ex1_offline}–\ref{tab:ex3_offline}. As expected from our numerical analysis, increasing the number of offline multiscale basis functions leads to a reduction in both pressure and velocity errors. Refining the coarse mesh (i.e., using a smaller coarse grid size $H$) also yields smaller errors. This is because using a finer coarse grid, while maintaining the same number of basis functions per coarse block, effectively increases the total dimension of the multiscale solution space, thereby improving accuracy. Another important observation is that for different coarse grid configurations, if the dimension of the multiscale space is the same, the corresponding errors are comparable. For example, in Table \ref{tab:ex1_offline}, a coarse grid of size $H$ (multiscale space dimension 54,574) yields a pressure error of $0.86\%$, whereas a refined coarse grid of size $H/2$ (dimension 52,397) gives a similar pressure error of $1.17\%$. Likewise, for multiscale space dimensions 108,772 ($H$) and 104,620 ($H/2$), the pressure errors are $0.33\%$ and $0.42\%$, respectively. The velocity error $e_{\mathbf{u}}$ exhibits the same trend. Similar trends are observed in cases 2 and 3, as shown in Tables \ref{tab:ex2_offline} and \ref{tab:ex3_offline}.

\begin{table}[htbp]
\centering
% n=20
\begin{tabular}{|c|c|c|}
\hline
\multicolumn{3}{|c|}{$H$} \\ \hline
Dim & $e_p$ & $e_{\mathbf{u}}$  \\ \hline
1,377 & 1.50e-01 & 3.35e+00 \\ \hline
13,706 & 3.38e-02 & 7.11e-01 \\ \hline
27,361 & 1.80e-02 & 4.23e-01 \\ \hline
54,574 & 8.62e-03 & 2.49e-01 \\ \hline
108,772 & 3.33e-03 & 1.32e-01 \\ \hline
\end{tabular}
% n=10
\begin{tabular}{|c|c|c|}
\hline
\multicolumn{3}{|c|}{$H/2$} \\ \hline
Dim & $e_p$ & $e_{\mathbf{u}}$  \\ \hline
5,261 & 9.06e-02 & 2.51e+00 \\ \hline
52,397 & 1.17e-02 & 3.57e-01 \\ \hline
104,620 & 4.18e-03 & 1.73e-01 \\ \hline
208,428 & 1.19e-03 & 7.37e-02 \\ \hline
413,107 & 1.27e-04 & 1.13e-02 \\ \hline
\end{tabular}
\caption{Errors for different numbers of offline basis functions and coarse grids in Model 1.}
\label{tab:ex1_offline}
\end{table}

\begin{table}[htbp]
\centering
% n=20
\begin{tabular}{|c|c|c|}
\hline
\multicolumn{3}{|c|}{$H$} \\ \hline
Dim & $e_p$ & $e_{\mathbf{u}}$  \\ \hline
770 & 1.79e-01 & 3.54e+00 \\ \hline
7,586 & 1.03e-02 & 6.30e-01 \\ \hline
15,101 & 2.24e-02 & 4.37e-01 \\ \hline
29,880 & 1.07e-02 & 2.80e-01 \\ \hline
58,860 & 4.42e-03 & 1.54e-01 \\ \hline
\end{tabular}
% n=10
\begin{tabular}{|c|c|c|}
\hline
\multicolumn{3}{|c|}{$H/2$} \\ \hline
Dim & $e_p$ & $e_{\mathbf{u}}$  \\ \hline
2,918 & 1.06-01 & 2.59e+00 \\ \hline
28,795 & 1.39e-02 & 3.43e-01 \\ \hline
57,199 & 5.05e-03 & 1.86e-01 \\ \hline
113,263 & 1.64e-03 & 7.77e-02 \\ \hline
221,767 & 1.77e-04 & 1.34e-02 \\ \hline
\end{tabular}
\caption{Errors for different numbers of offline basis functions and coarse grids in Model 2.}
\label{tab:ex2_offline}
\end{table}

\begin{table}[htbp]
\centering
% n=20
\begin{tabular}{|c|c|c|}
\hline
\multicolumn{3}{|c|}{$H$} \\ \hline
Dim & $e_p$ & $e_{\mathbf{u}}$  \\ \hline
1,411 & 3.19e-01 & 2.20e+00 \\ \hline
13,246 & 9.84e-02 & 7.16e-01 \\ \hline
25,499 & 5.37e-02 & 4.84e-01 \\ \hline
48,994 & 2.69e-02 & 3.14e-01 \\ \hline
93,529 & 1.14e-02 & 1.83e-01 \\ \hline
\end{tabular}
% n=10
\begin{tabular}{|c|c|c|}
\hline
\multicolumn{3}{|c|}{$H/2$} \\ \hline
Dim & $e_p$ & $e_{\mathbf{u}}$  \\ \hline
4,532 & 2.34e-01 & 1.82e+00 \\ \hline
43,051 & 3.58e-02 & 3.96e-01 \\ \hline
83,420 & 1.48e-02 & 2.21e-01 \\ \hline
159,352 & 5.62e-03 & 1.16e-01 \\ \hline
291,830 & 7.82e-04 & 2.19e-02 \\ \hline
\end{tabular}
\caption{Errors for different numbers of offline basis functions and coarse grids in Model 3.}
\label{tab:ex3_offline}
\end{table}

\subsection{Offline bases vs.\ Online bases}

In this subsection, we compare the efficiency of offline and online multiscale basis functions. Starting with a single offline basis function per coarse block, we uniformly enrich the multiscale space by successively adding either offline or online basis functions over nine iterations. To examine the influence of grid resolution, we perform experiments on different coarse mesh sizes. The results for all three media are presented in Figure \ref{fig:er_offon}.
From this figure, we observe that for a given multiscale space dimension, the resulting errors are similar across different coarse grids. This observation is consistent with the findings from the previous subsection. In addition, the error decay achieved through online enrichment is significantly faster than that obtained via offline enrichment, highlighting the superior efficiency of the online approach.

For all three test cases, we find that with a coarse mesh size of $H$, the pressure error drops to the order of $10^{-3}$ after nine steps of online enrichment, whereas the corresponding error for offline enrichment remains at the order of $10^{-2}$. When the coarse mesh is refined to $H/2$, the final error of pressure is approximately $10^{-2}$ for offline enrichment and $10^{-4}$ for online enrichment.
Moreover, for a fixed multiscale space dimension, using a coarser mesh generally leads to a more accurate solution.
This improvement is due to the fact that, under a fixed global multiscale dimension, coarser grids allow a larger portion of the degrees of freedom to be dedicated to constructing online basis functions within each coarse block, thereby enabling more effective capture the influence of the source term.

We also report the computational cost for both offline and online enrichments on the coarse grid with mesh size $H$ in Table~\ref{tab:cputime_offon}. 
All computations were performed on a standard laptop (Intel Core i7 CPU, 32~GB RAM), and the reported times correspond to solving the coarse-scale linear systems associated with the multiscale formulations. 
For the offline enrichment, the reported times correspond to solving the coarse-scale system for different total numbers of multiscale basis functions per coarse block ($L=1,4,7,10$), without any iterative procedure. In other words, the offline computation does not involve repeated coarse-scale solves as the basis functions are fixed. 
In contrast, the online enrichment procedure involves multiple iterations, with one coarse-scale problem solved at each iteration. In Table~\ref{tab:cputime_offon}, the online timings correspond to 3, 6, and 9 iterations, where new online basis functions are sequentially added to enrich the multiscale space. Naturally, this iterative procedure increases the overall computational cost. 
Nevertheless, the online approach achieves a much faster reduction of the approximation error for a given multiscale space dimension, reflecting its effectiveness in improving solution accuracy relative to the offline-only approach.

\begin{table}[H]
\centering
\begin{tabular}{|c|c|c|c|c|c|c|}
\hline
\multirow{2}*{}
& \multicolumn{2}{|c|}{Model 1} 
& \multicolumn{2}{|c|}{Model 2} 
& \multicolumn{2}{|c|}{Model 3} \\ \cline{2-7}
& Offline & Online & Offline & Online & Offline & Online 
\\ \hline
1 & 0.0056 & 0.0056 & 0.0037 & 0.0037 & 0.0046 & 0.0046 \\ \hline
4 & 0.0277 & 0.0919 & 0.0131 & 0.0330 & 0.0283 & 0.0537 \\ \hline
7 & 0.0665 & 0.2614 & 0.0363 & 0.1109 & 0.0551 & 0.1721 \\ \hline
10 & 0.1636 & 0.6397 & 0.0633 & 0.2661 & 0.0881 & 0.3982 \\ \hline
\end{tabular}
\caption{CPU time (in seconds) for three models with different numbers of enrichment steps. 
The fine-grid solver times for Models~1--3 are 1.2039~s, 0.6499~s, and 0.7563~s, respectively.}
\label{tab:cputime_offon}
\end{table}

\begin{figure}[htbp]
\centering
\begin{subfigure}[b]{0.49\textwidth}
\includegraphics[width=\linewidth]{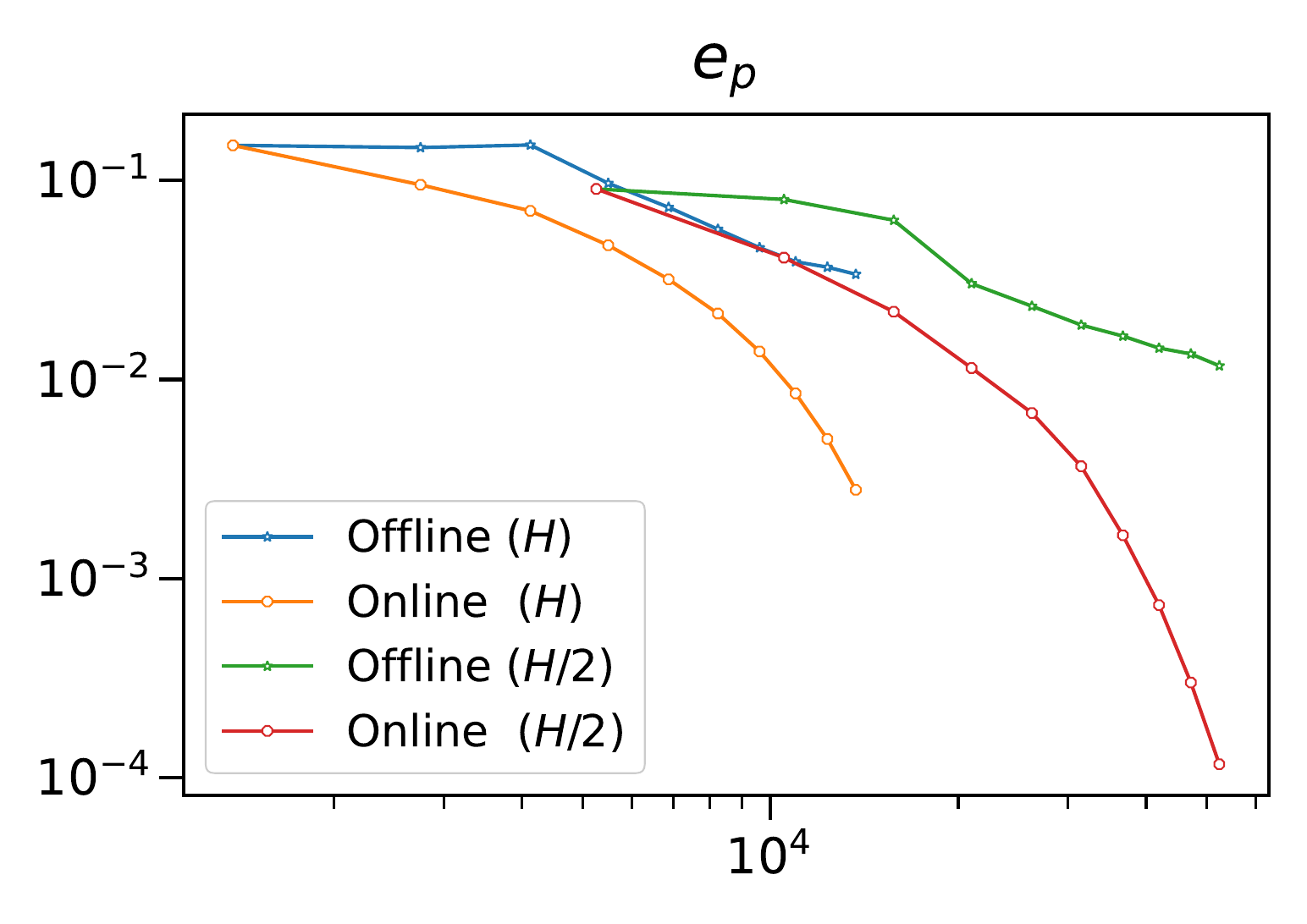}
\caption{}
\end{subfigure}
\begin{subfigure}[b]{0.49\textwidth}
\includegraphics[width=\linewidth]{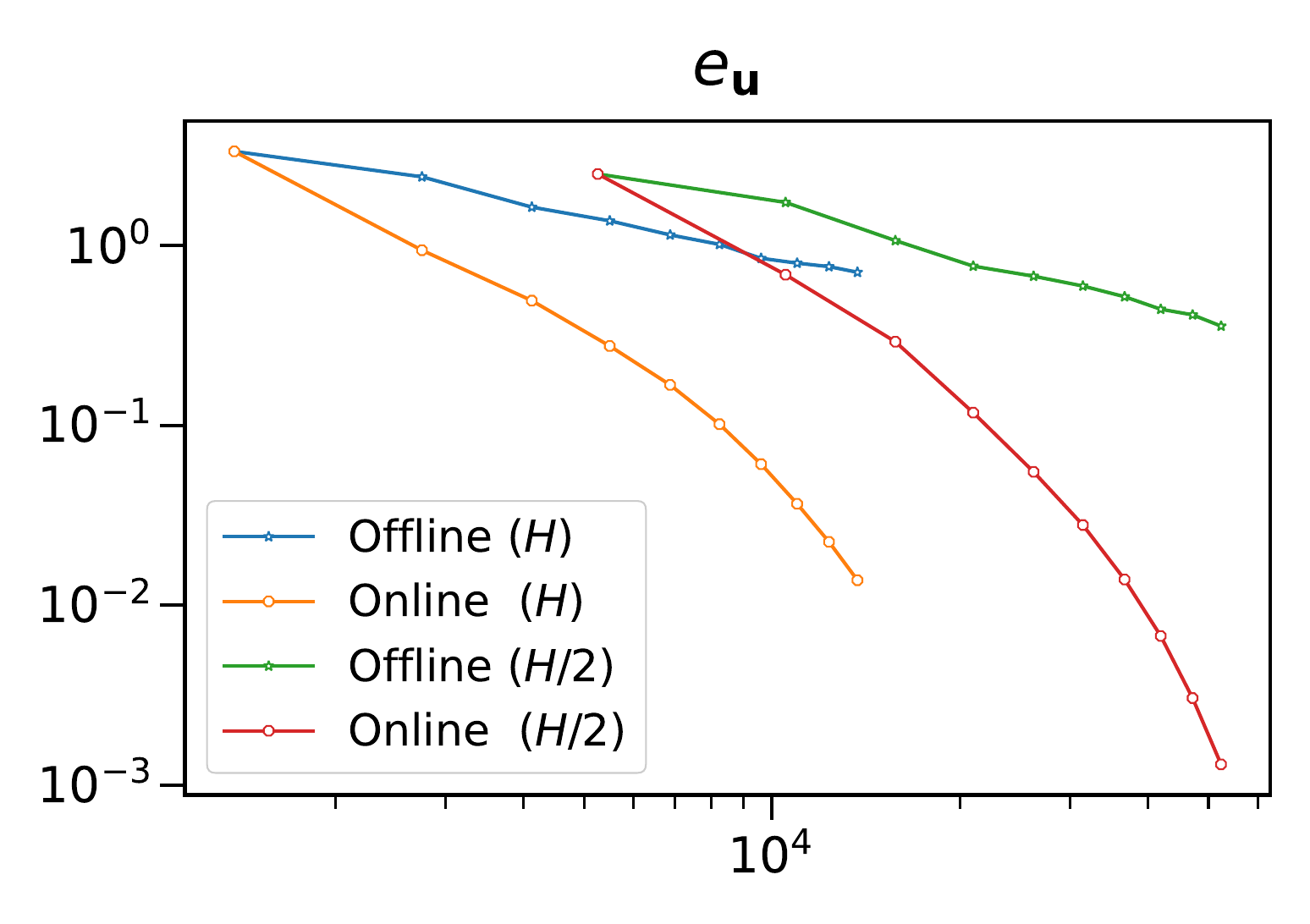}
\caption{}
\end{subfigure}
\begin{subfigure}[b]{0.49\textwidth}
\includegraphics[width=\linewidth]{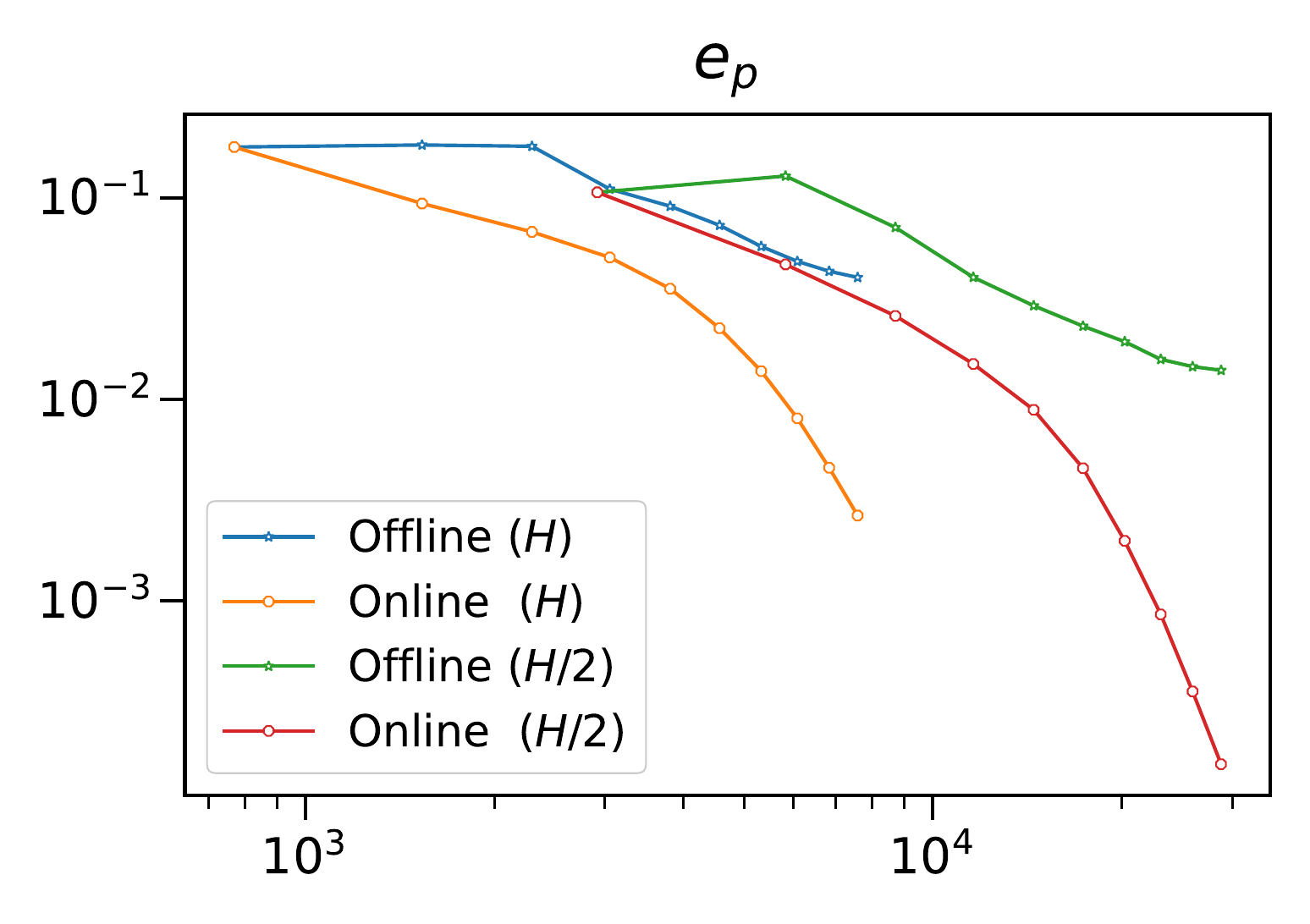}
\caption{}
\end{subfigure}
\begin{subfigure}[b]{0.49\textwidth}
\includegraphics[width=\linewidth]{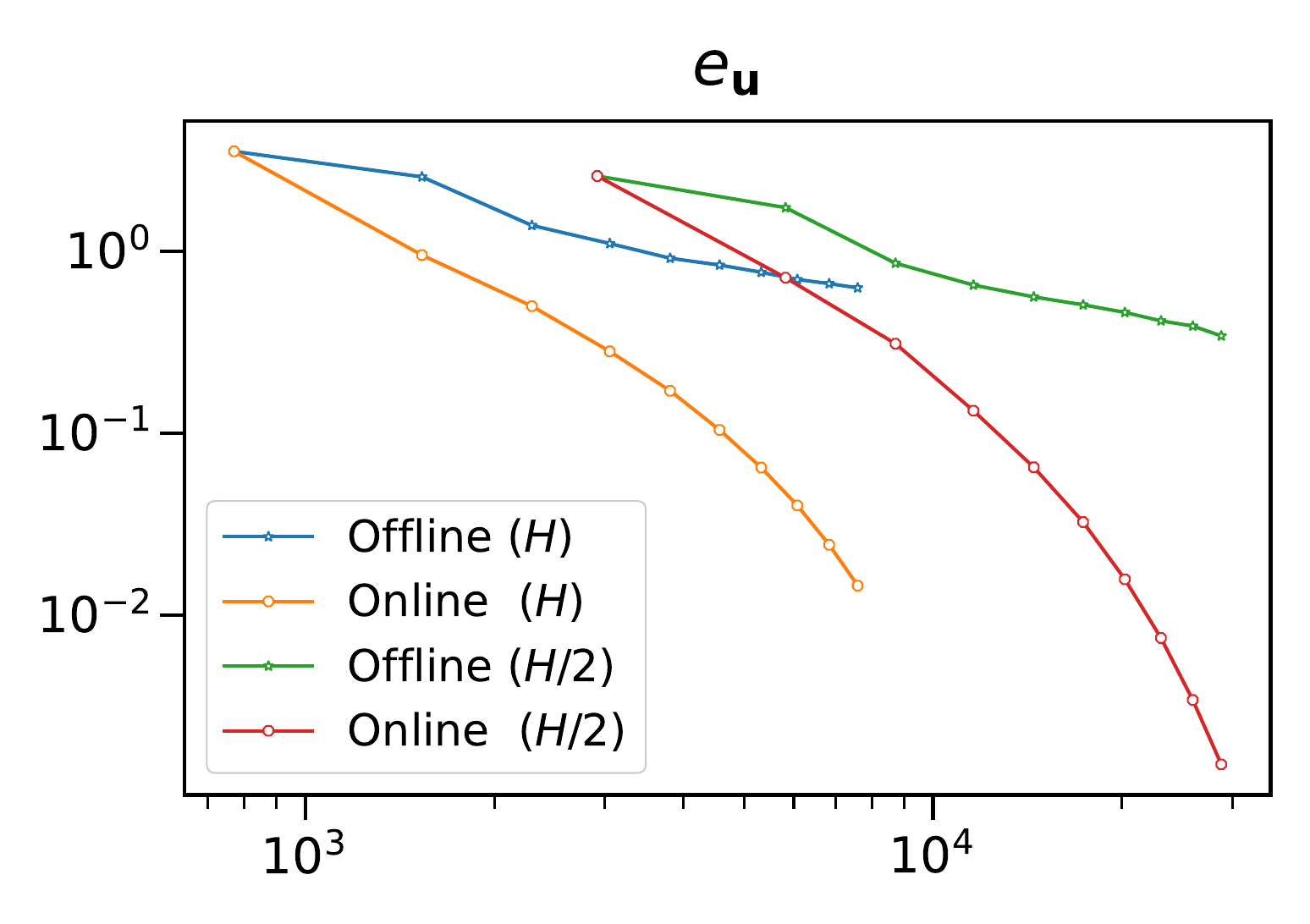}
\caption{}
\end{subfigure}
\begin{subfigure}[b]{0.49\textwidth}
\includegraphics[width=\linewidth]{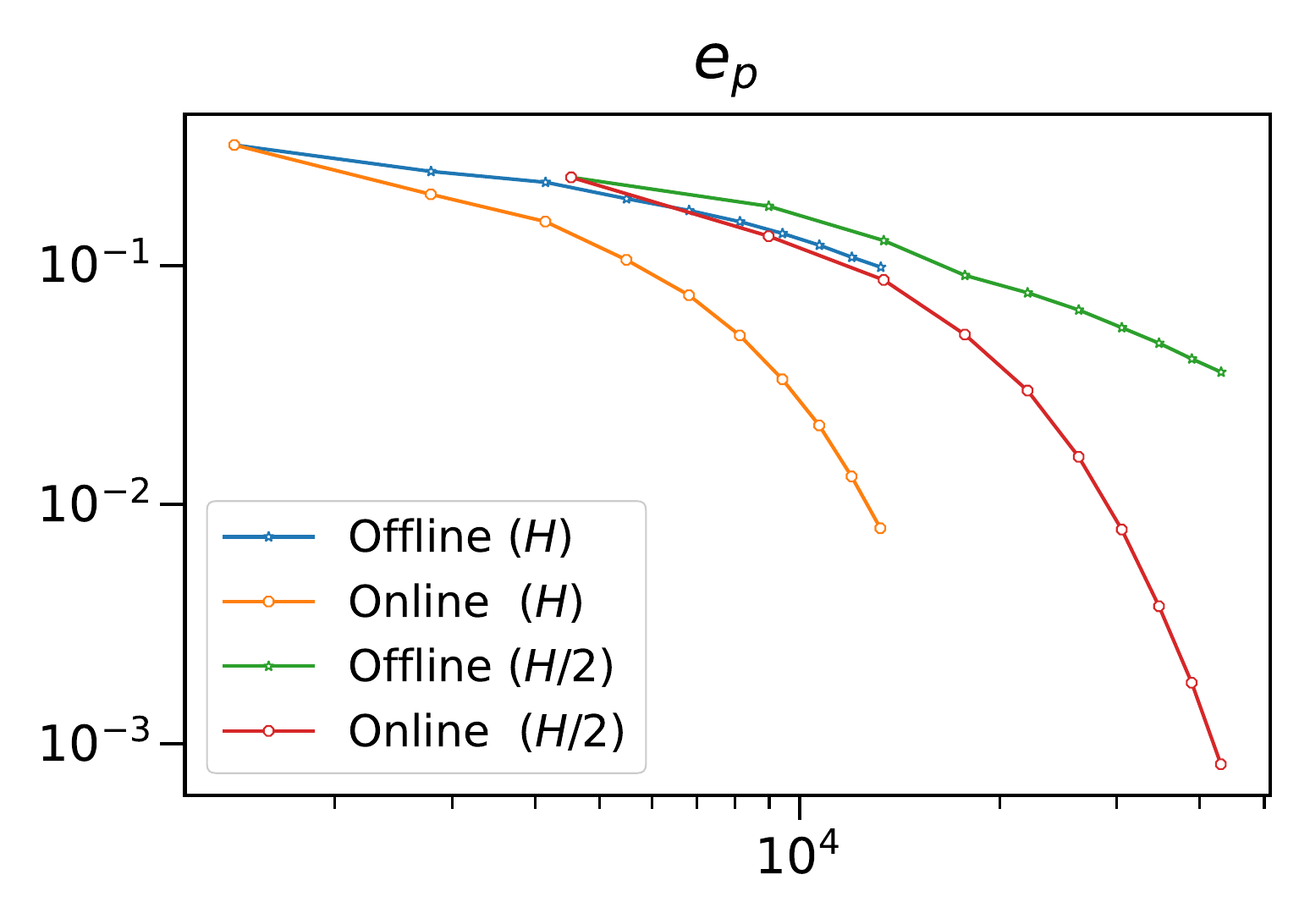}
\caption{}
\end{subfigure}
\begin{subfigure}[b]{0.49\textwidth}
\includegraphics[width=\linewidth]{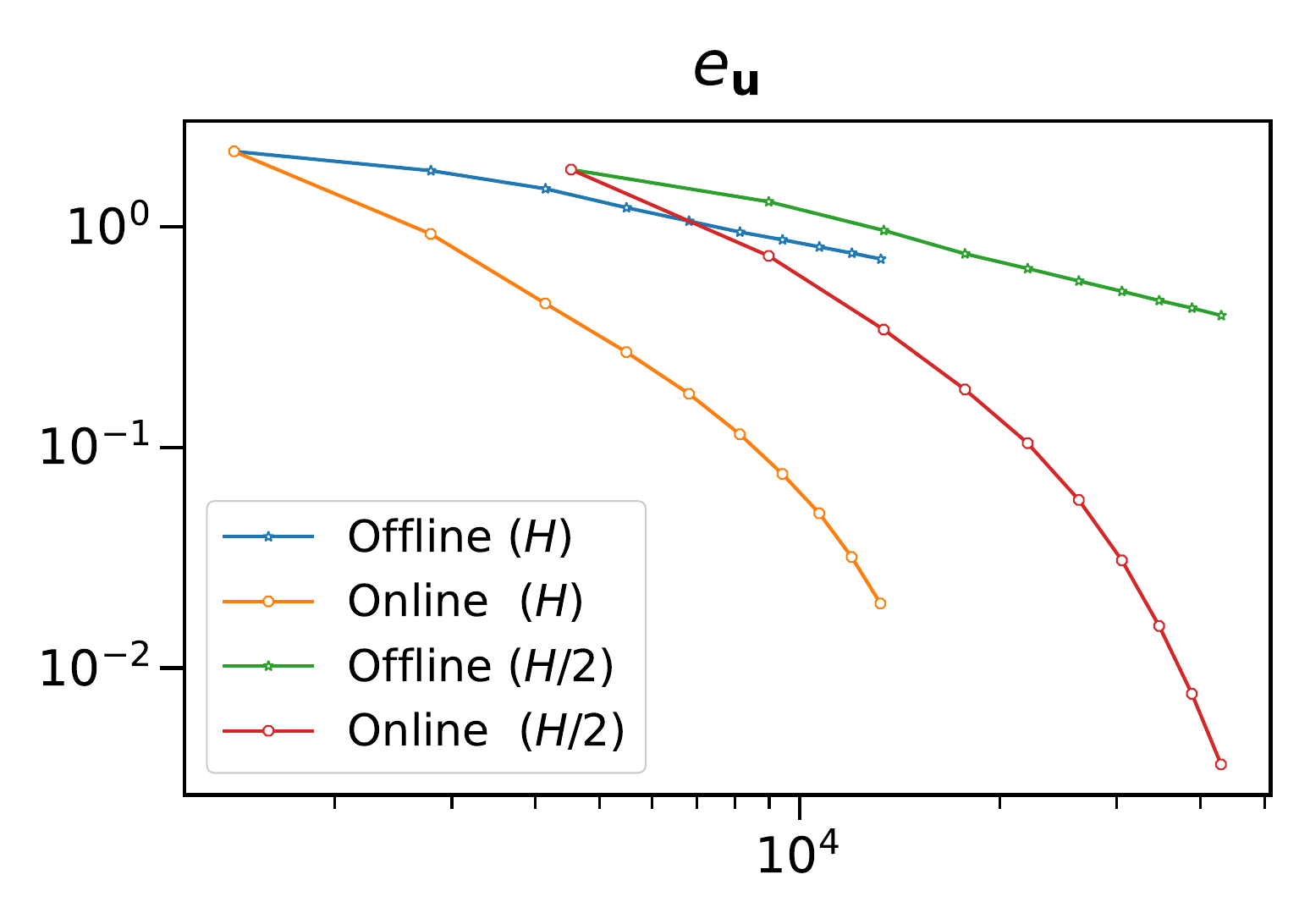}
\caption{}
\end{subfigure}
\caption{Errors with offline and online enrichment for three models: (a, b) Model 1; (c, d) Model 2; (e, f) Model 3.}
\label{fig:er_offon}
\end{figure}

\subsection{Effect of Offline Basis Number and Adaptive Parameter \texorpdfstring{$\theta$}{theta}}

\begin{figure}[htbp]
\centering
\begin{subfigure}[b]{0.49\textwidth}
\includegraphics[width=\linewidth]{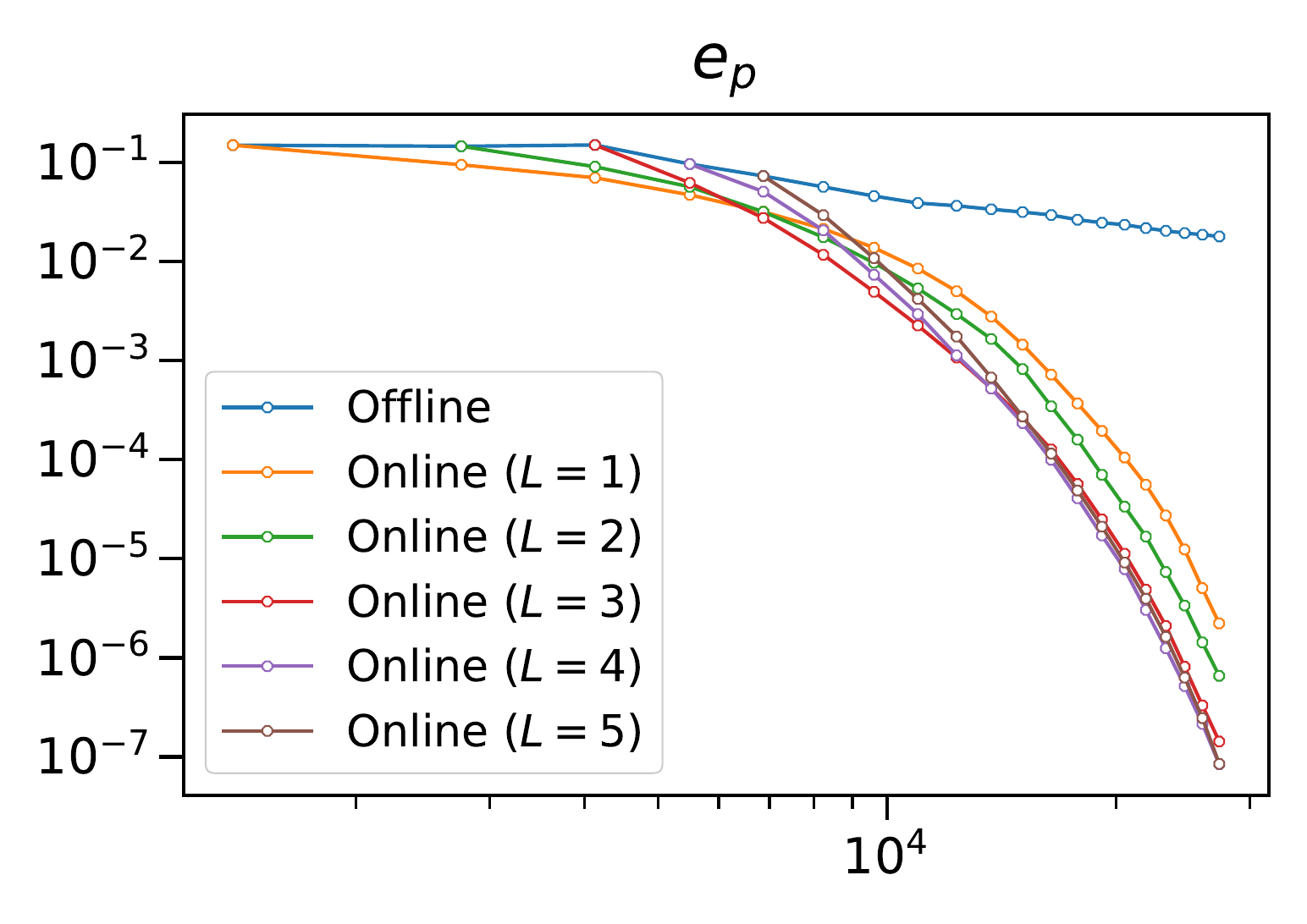}
\caption{}
\end{subfigure}
\begin{subfigure}[b]{0.49\textwidth}
\includegraphics[width=\linewidth]{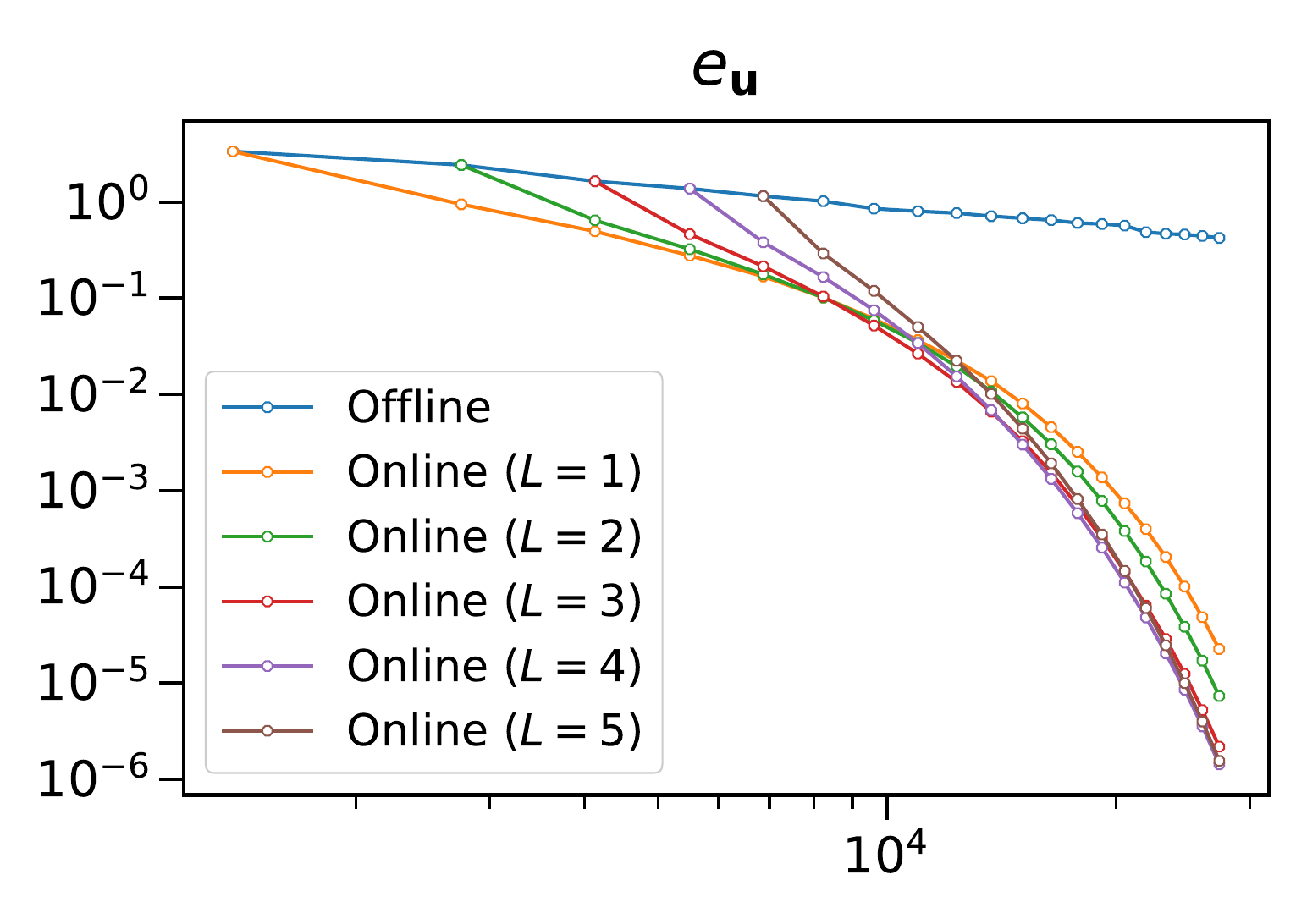}
\caption{}
\end{subfigure}
\begin{subfigure}[b]{0.49\textwidth}
\includegraphics[width=\linewidth]{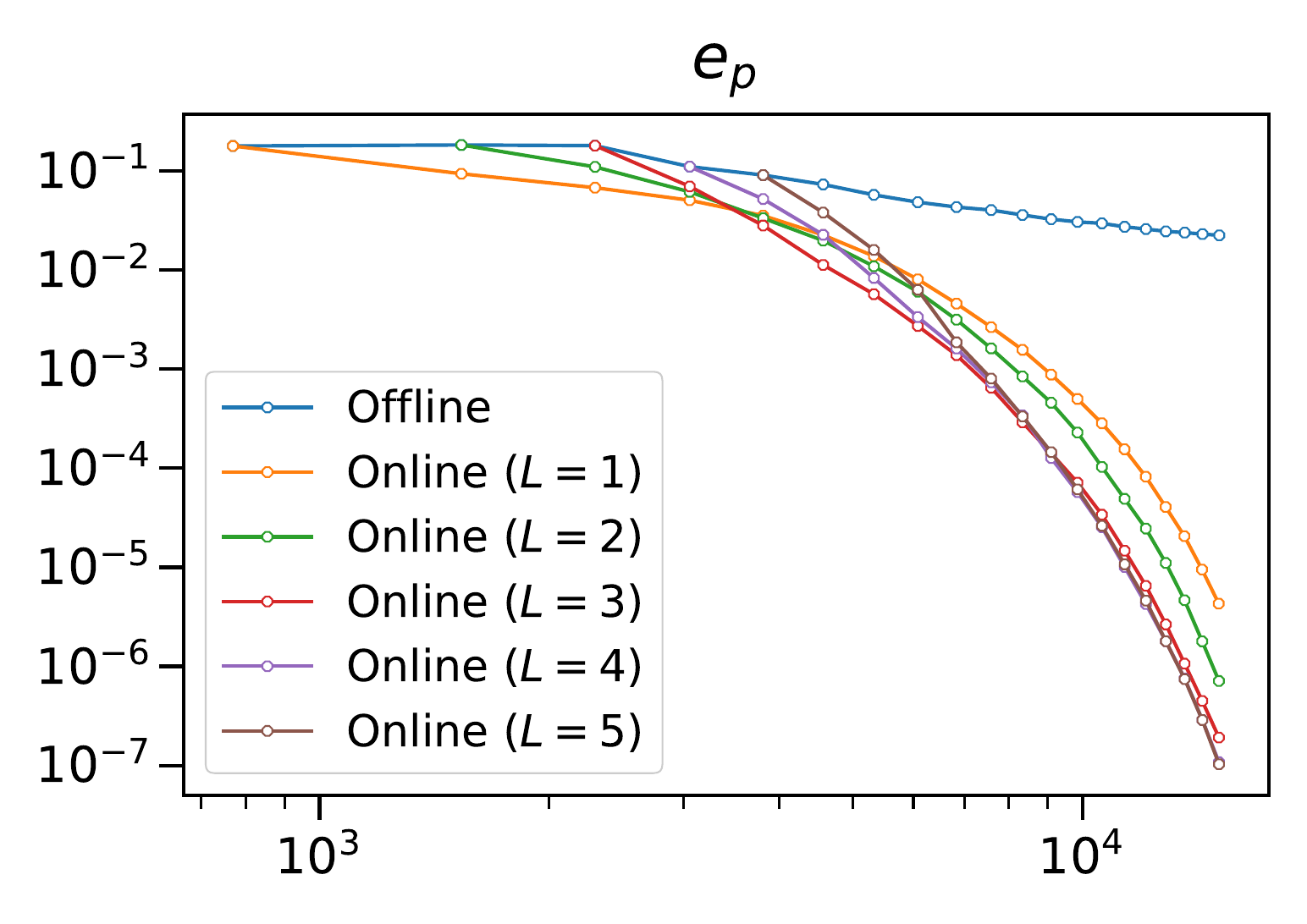}
\caption{}
\end{subfigure}
\begin{subfigure}[b]{0.49\textwidth}
\includegraphics[width=\linewidth]{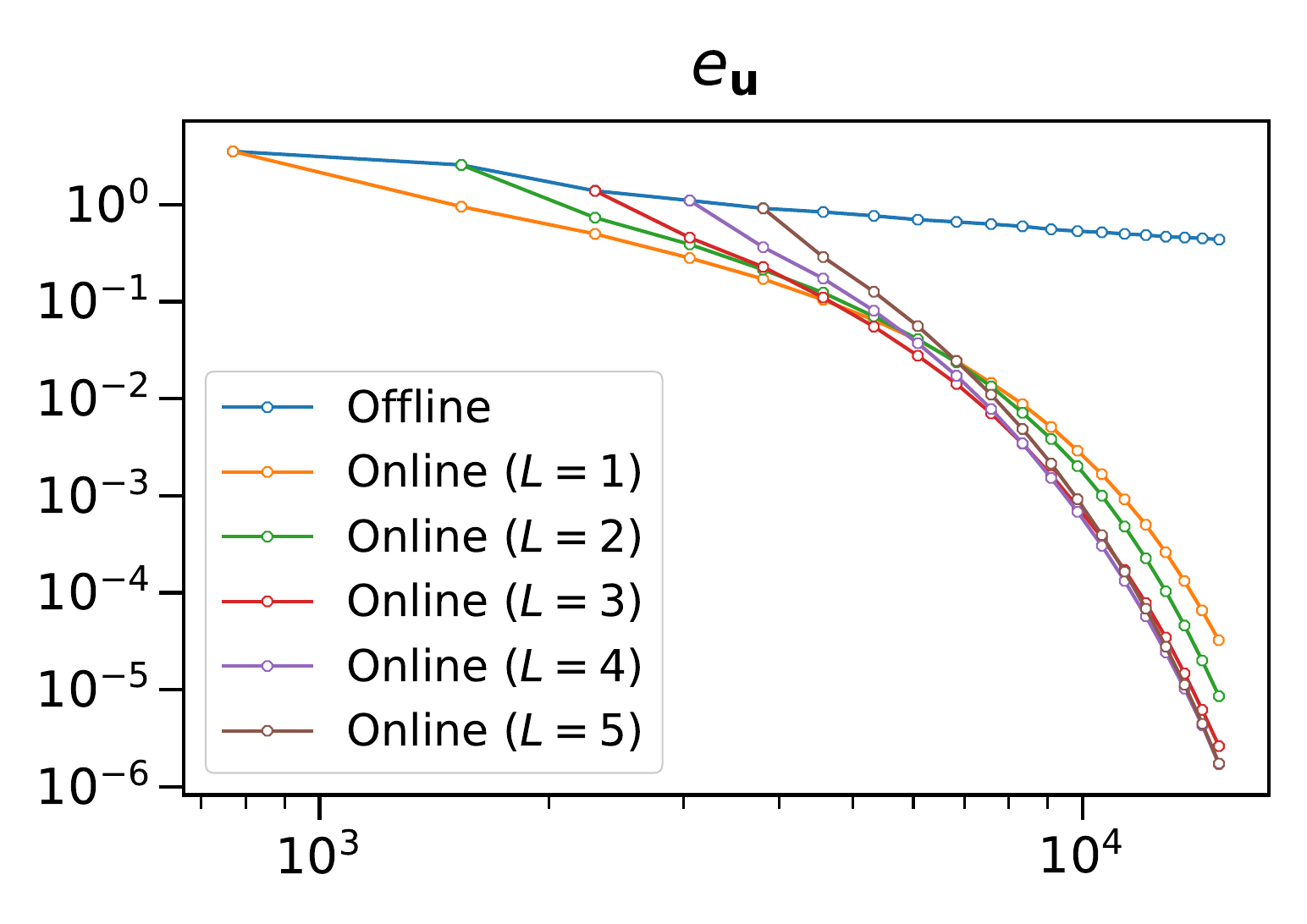}
\caption{}
\end{subfigure}
\begin{subfigure}[b]{0.49\textwidth}
\includegraphics[width=\linewidth]{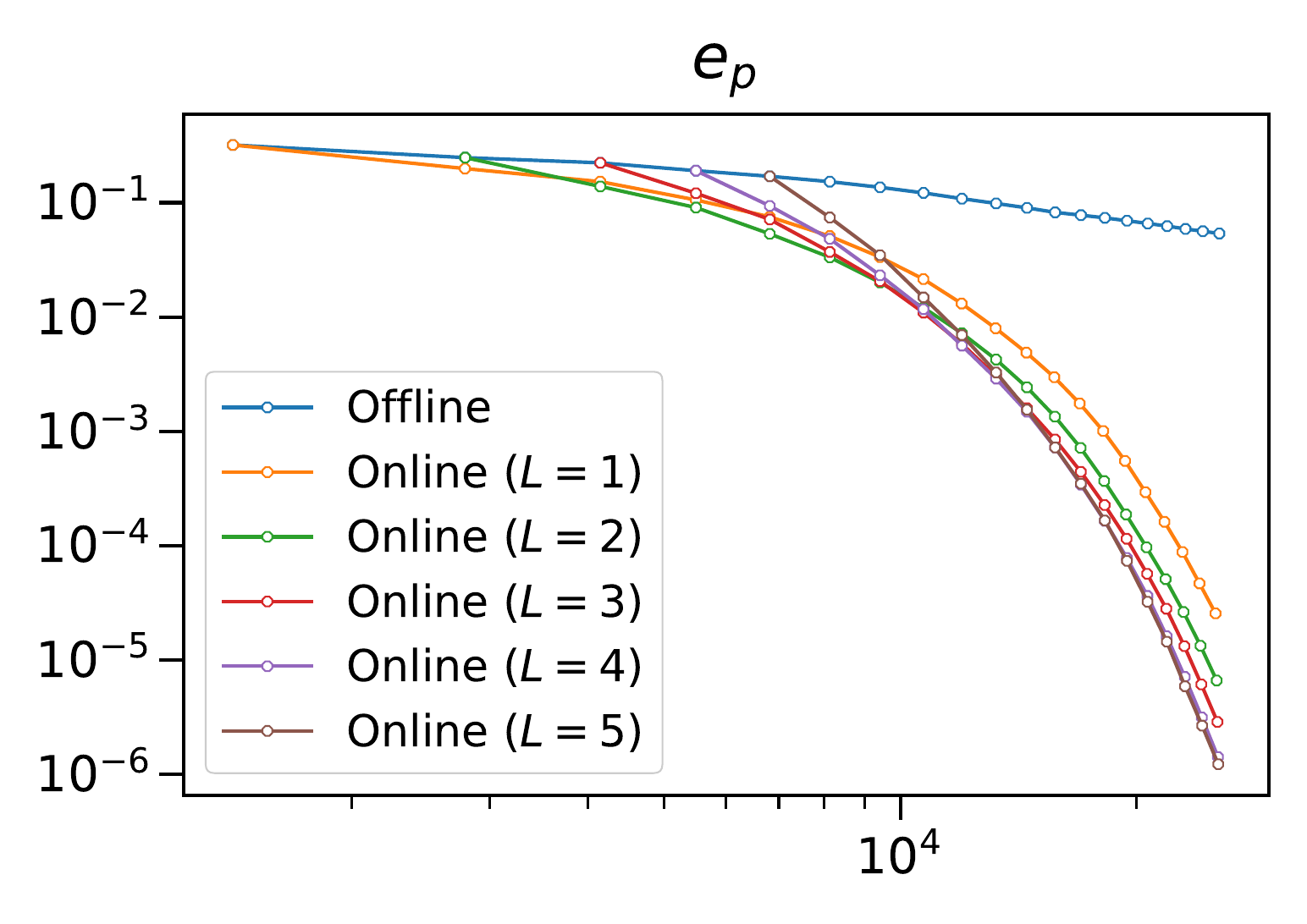}
\caption{}
\end{subfigure}
\begin{subfigure}[b]{0.49\textwidth}
\includegraphics[width=\linewidth]{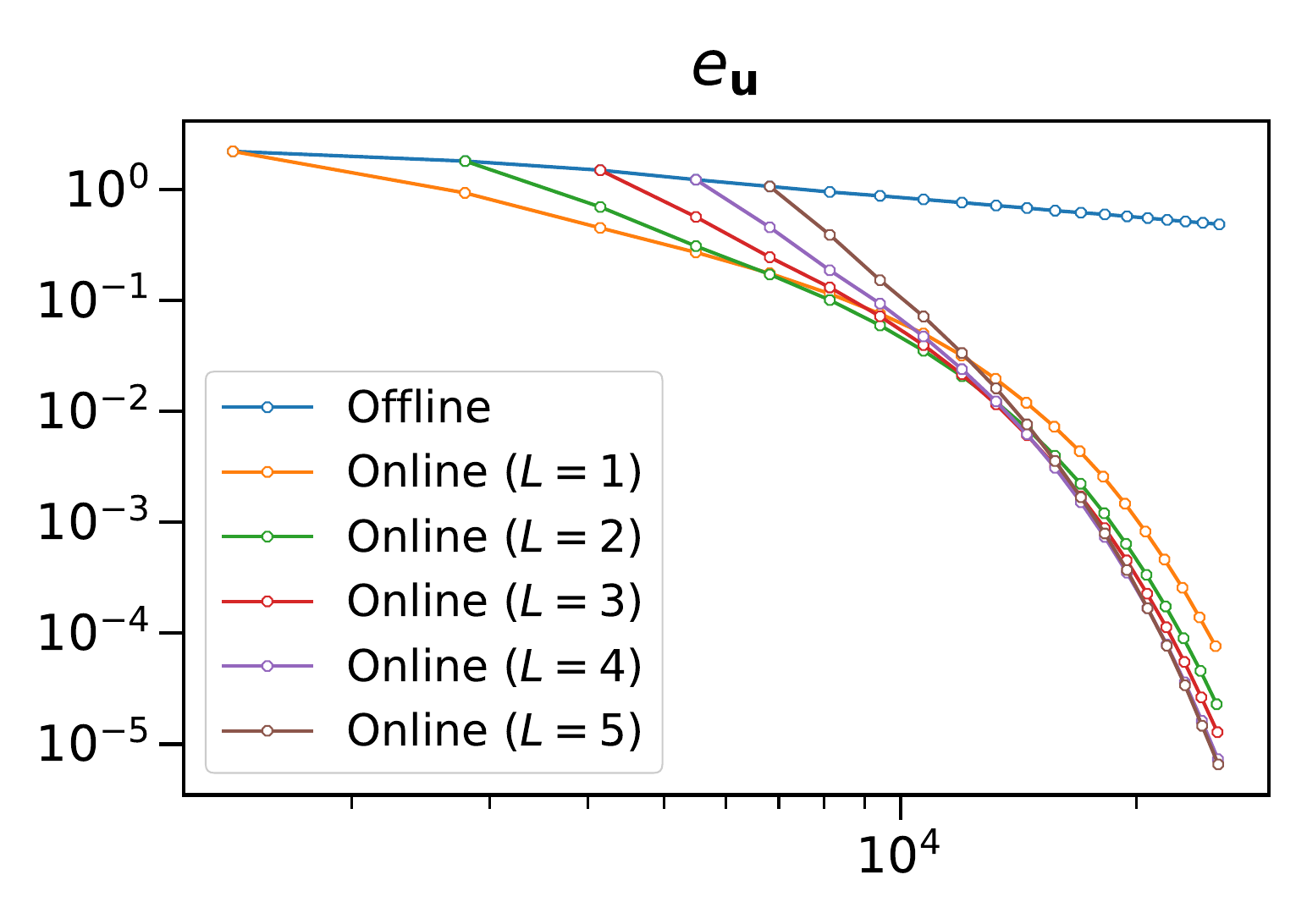}
\caption{}
\end{subfigure}
\caption{Errors with different numbers of offline basis functions $L$ in the online iteration for three models: (a, b) Model 1; (c, d) Model 2; (e, f) Model 3.}
\label{fig:er_onli}
\end{figure}

In this subsection, we investigate the influence of the number of offline basis functions $L$ and the adaptive parameter $\theta$ on the performance of the proposed method, under a fixed coarse grid size $H$.

We first examine how varying the number of offline basis functions affects the solution. Specifically, we consider five initial offline spaces with $L = 1, 2, \dots, 5$ basis functions per coarse block. In each case, we subsequently perform $20 - L$ steps of uniform online enrichment. The corresponding results are displayed in Figure \ref{fig:er_onli}, alongside the results from purely offline enrichment. As observed, the online enrichment consistently outperforms the offline strategy in terms of accuracy, corroborating the findings presented in the previous subsection. Furthermore, increasing the number of offline basis functions in the initial space improves the final accuracy of the online solution. However, the marginal improvement becomes negligible once three offline basis functions per coarse block are included.

Next, we explore the effect of different values of the adaptive parameter $\theta$. We consider four cases: $\theta = 1, 0.8, 0.6, 0.4$. Each test starts with three offline basis functions per coarse block, followed by $5$, $8$, $12$, and $17$ steps of online enrichment, respectively. The results are reported in Figure \ref{fig:er_ontheta}. As expected, uniform enrichment ($\theta = 1$) achieves the fastest convergence per iteration. However, adaptive enrichment leads to more accurate solutions for a fixed multiscale space dimension. This behavior is attributed to the selective allocation of online basis functions in regions with larger residuals, effectively reducing global error. This trend is consistently observed across all three media.
It is also observed that varying $\theta$ produces only minor differences in the overall convergence behavior.  
This phenomenon can be explained by the high efficiency of the online basis functions: once a region with large residuals is enriched, its error is rapidly reduced, leading to similar enrichment patterns in subsequent iterations even for different $\theta$ values.  
In practice, $\theta$ controls the trade-off between computational cost and adaptivity.  
Smaller $\theta$ values reduce the number of local problems solved per iteration, while larger values accelerate convergence at the expense of additional cost.  
A more systematic investigation of the optimal choice of $\theta$ and the tolerance parameter used for enrichment termination will be considered in future work.

Finally, we illustrate representative multiscale pressure solutions in Figures~\ref{fig:ex1_phms}--\ref{fig:ex3_phms}. The visual comparison shows no significant discrepancy between solutions, further confirming the robustness and efficiency of the proposed method. 

\begin{figure}[htbp]
\centering
\begin{subfigure}[b]{0.49\textwidth}
\includegraphics[width=\linewidth]{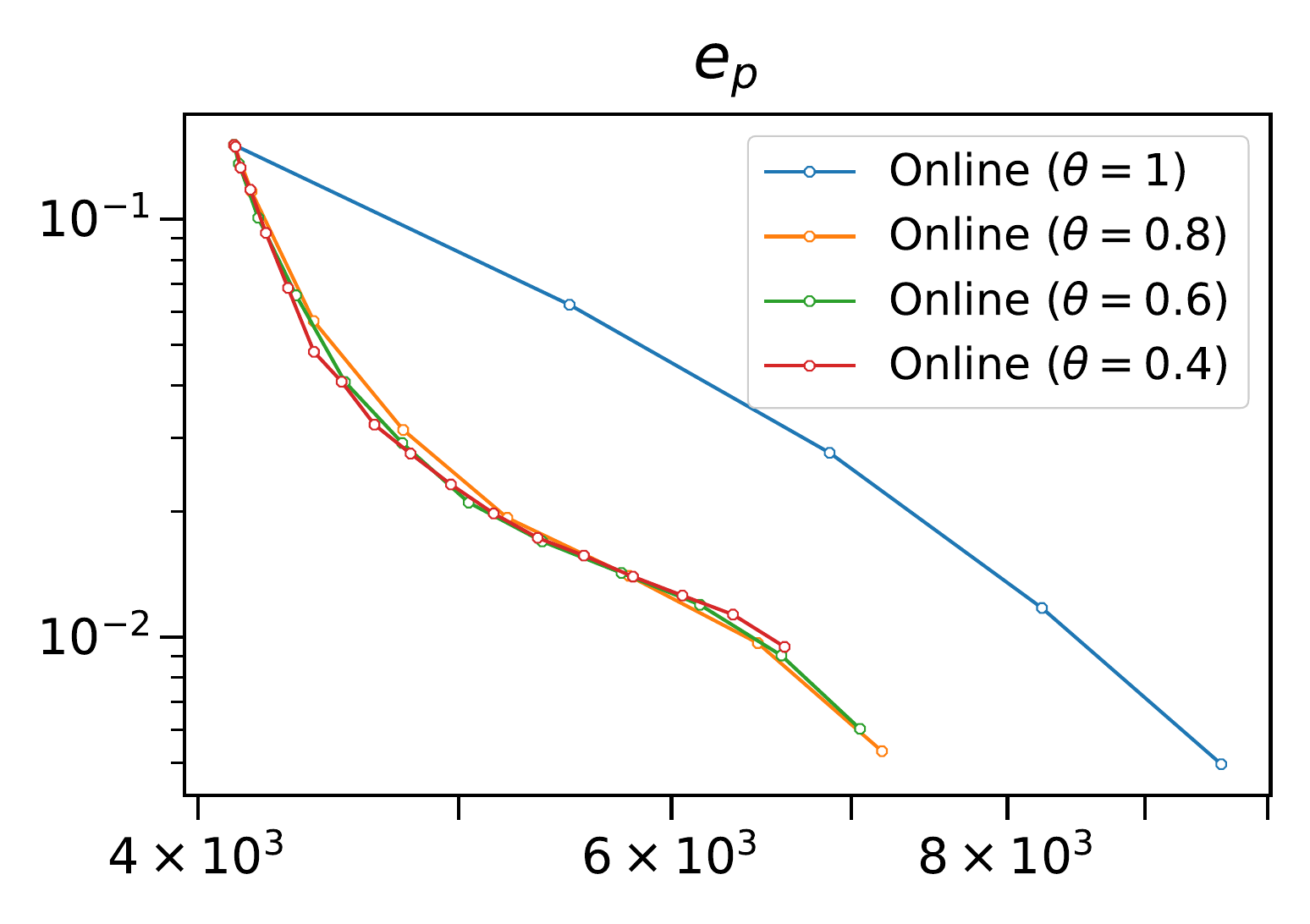}
\caption{}
\end{subfigure}
\begin{subfigure}[b]{0.49\textwidth}
\includegraphics[width=\linewidth]{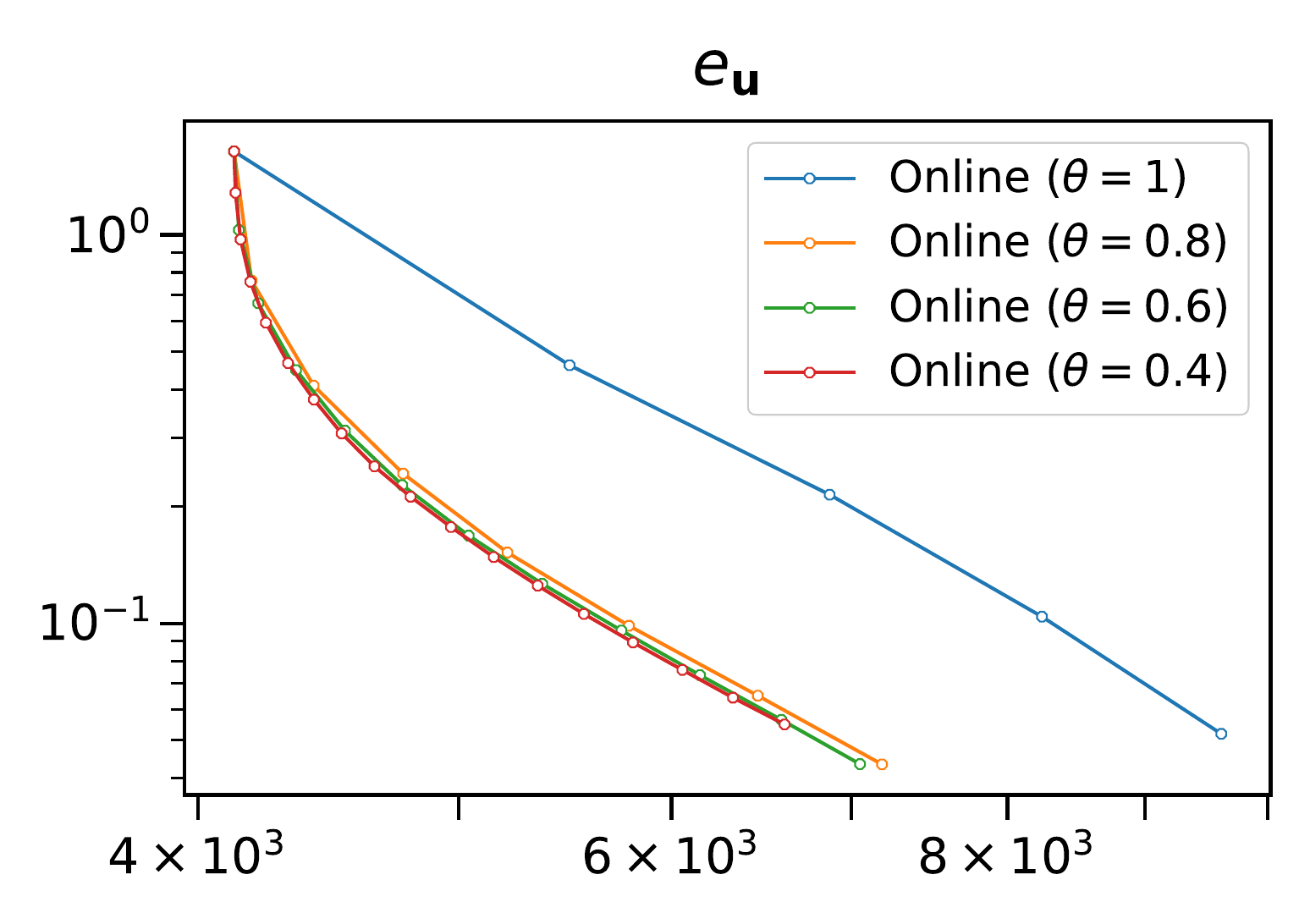}
\caption{}
\end{subfigure}
\begin{subfigure}[b]{0.49\textwidth}
\includegraphics[width=\linewidth]{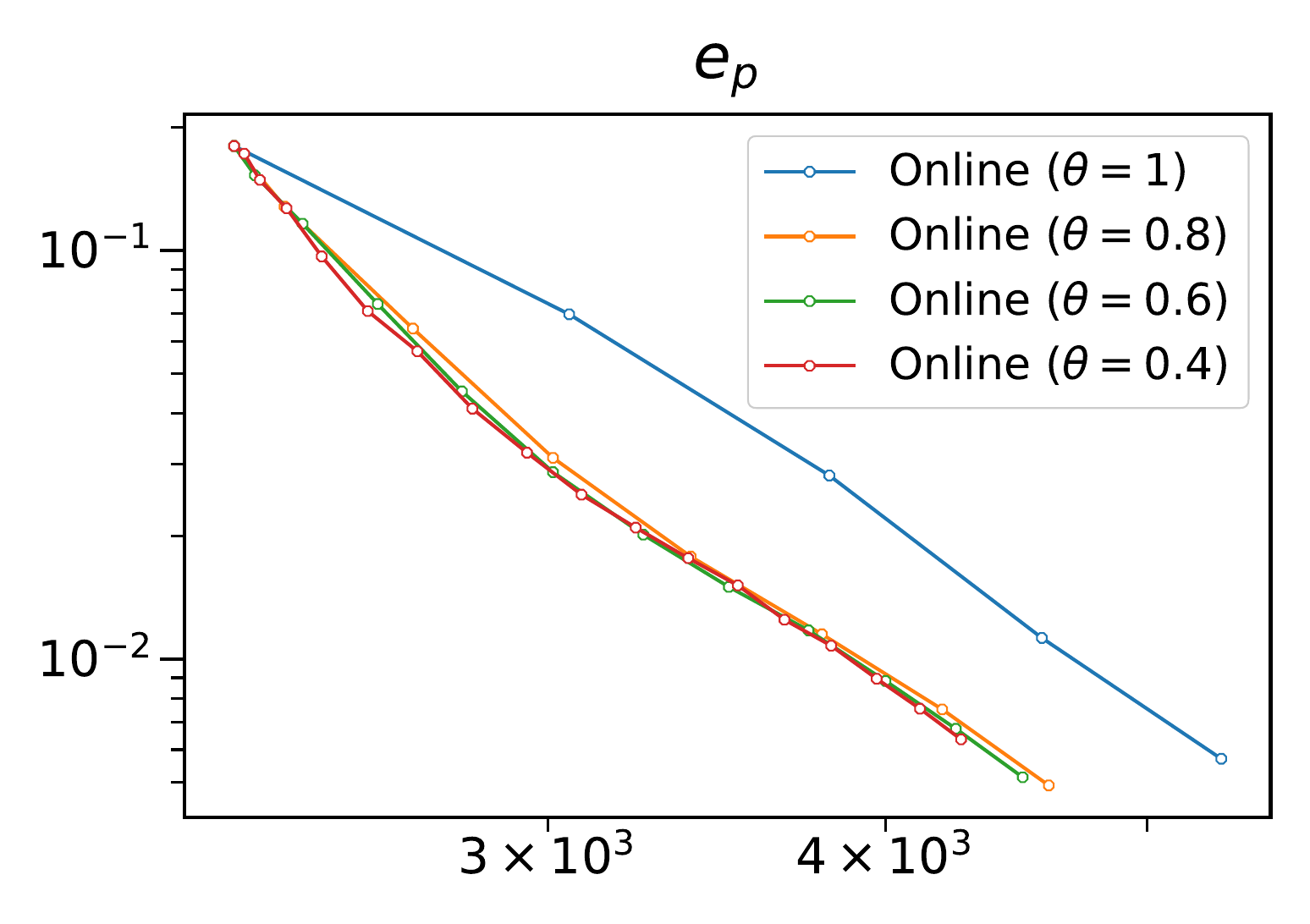}
\caption{}
\end{subfigure}
\begin{subfigure}[b]{0.49\textwidth}
\includegraphics[width=\linewidth]{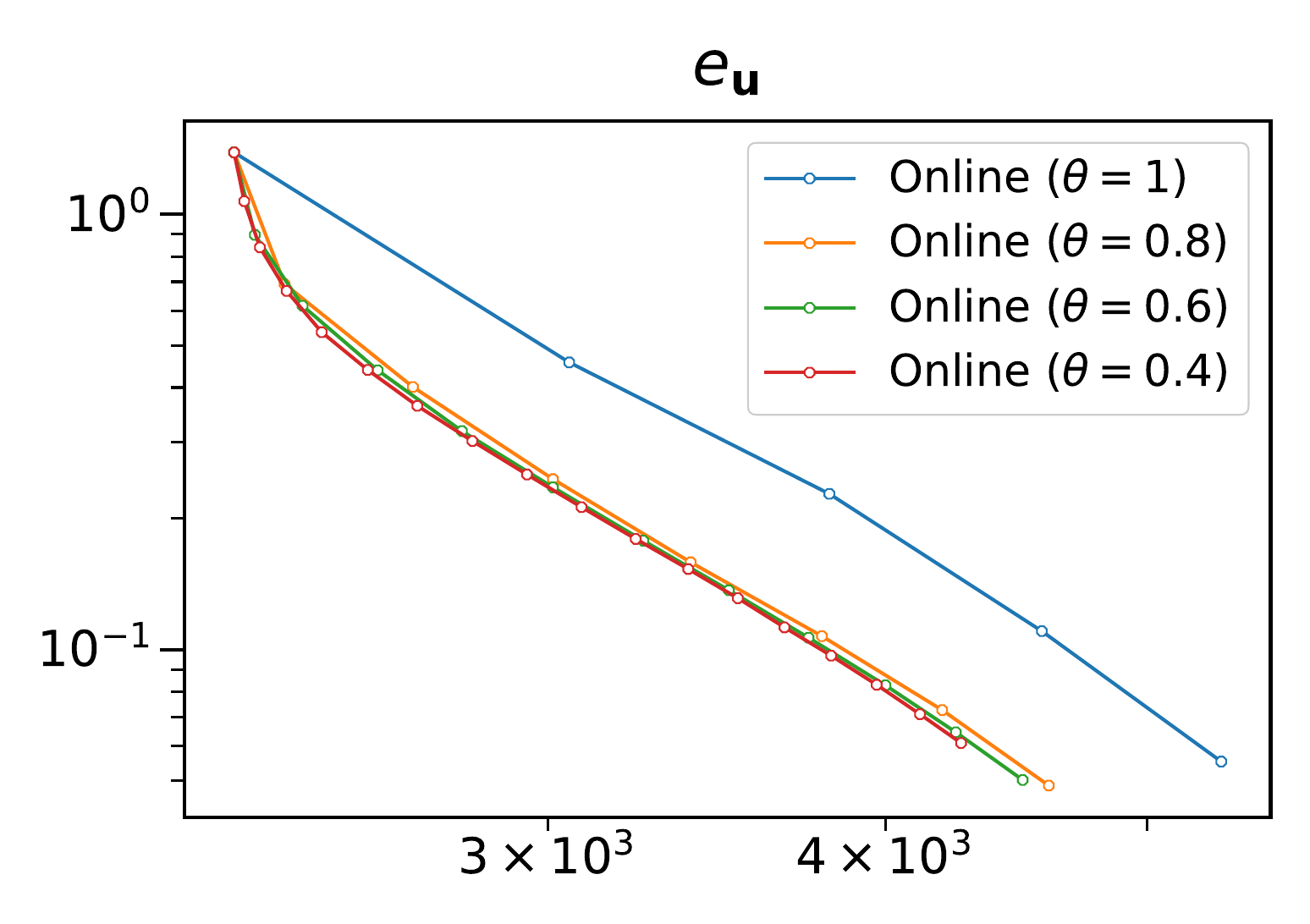}
\caption{}
\end{subfigure}
\begin{subfigure}[b]{0.49\textwidth}
\includegraphics[width=\linewidth]{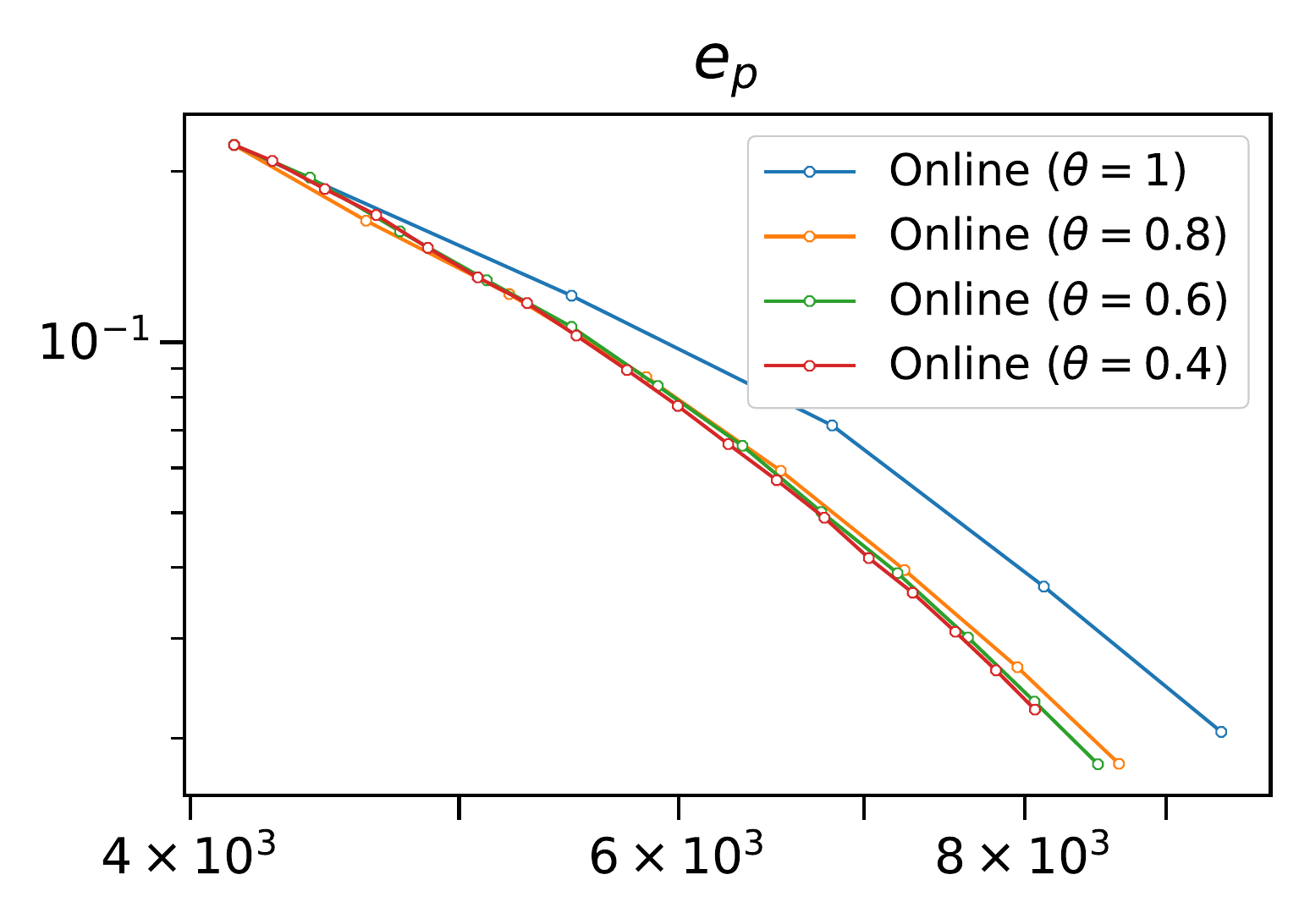}
\caption{}
\end{subfigure}
\begin{subfigure}[b]{0.49\textwidth}
\includegraphics[width=\linewidth]{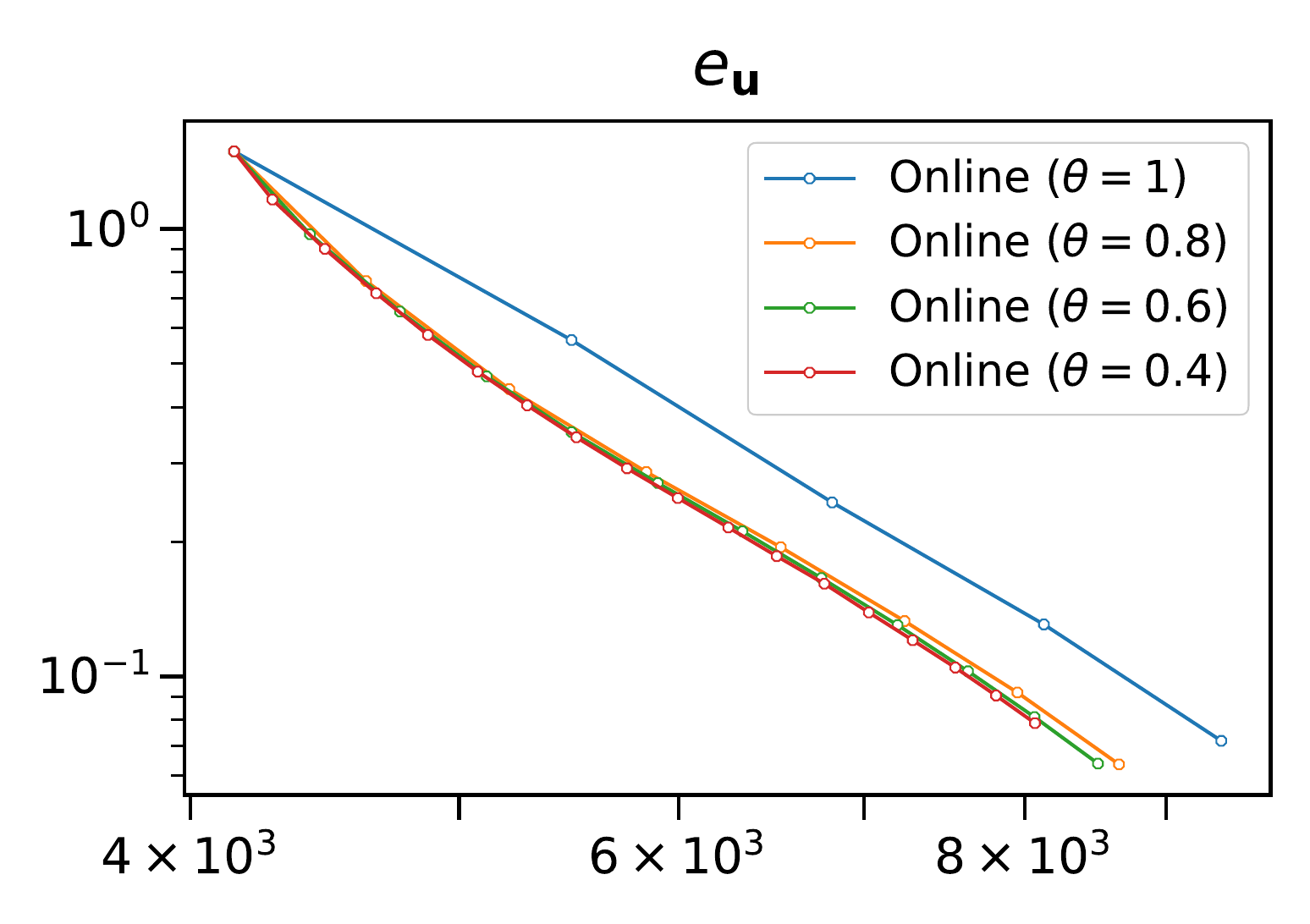}
\caption{}
\end{subfigure}
\caption{Errors with different adaptive parameter $\theta$ for three models: (a, b) Model 1; (c, d) Model 2; (e, f) Model 3.}
\label{fig:er_ontheta}
\end{figure}

\begin{figure}[htbp]
\centering
\begin{subfigure}[b]{0.32\textwidth}
\includegraphics[width=\linewidth]{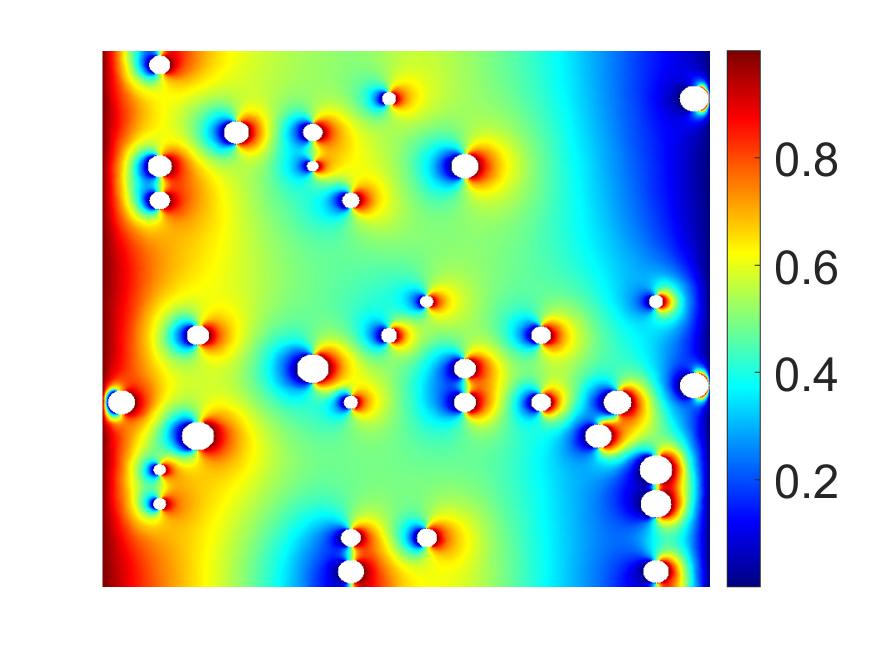}
\caption{}
\end{subfigure}
\begin{subfigure}[b]{0.32\textwidth}
\includegraphics[width=\linewidth]{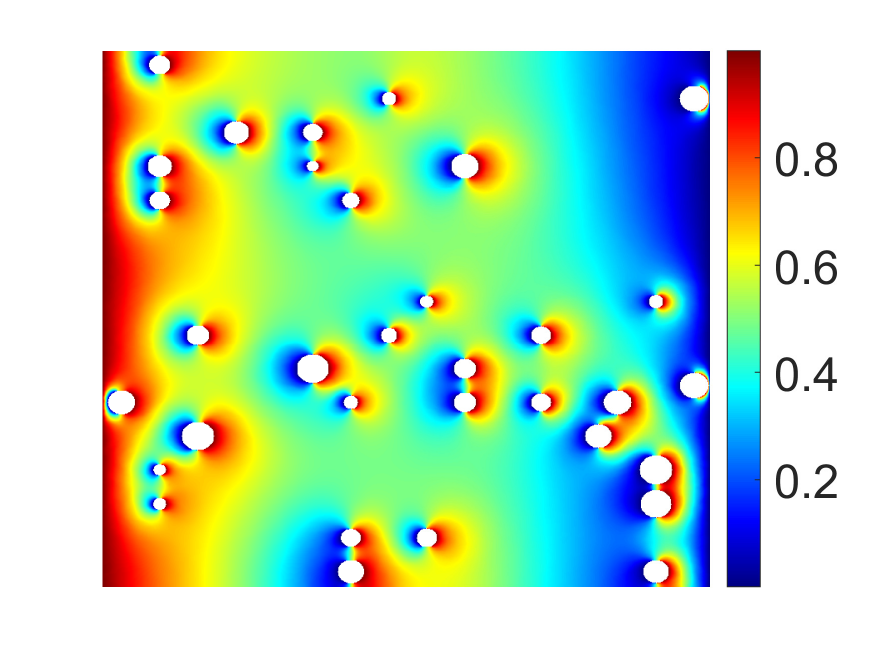}
\caption{}
\end{subfigure}
\begin{subfigure}[b]{0.32\textwidth}
\includegraphics[width=\linewidth]{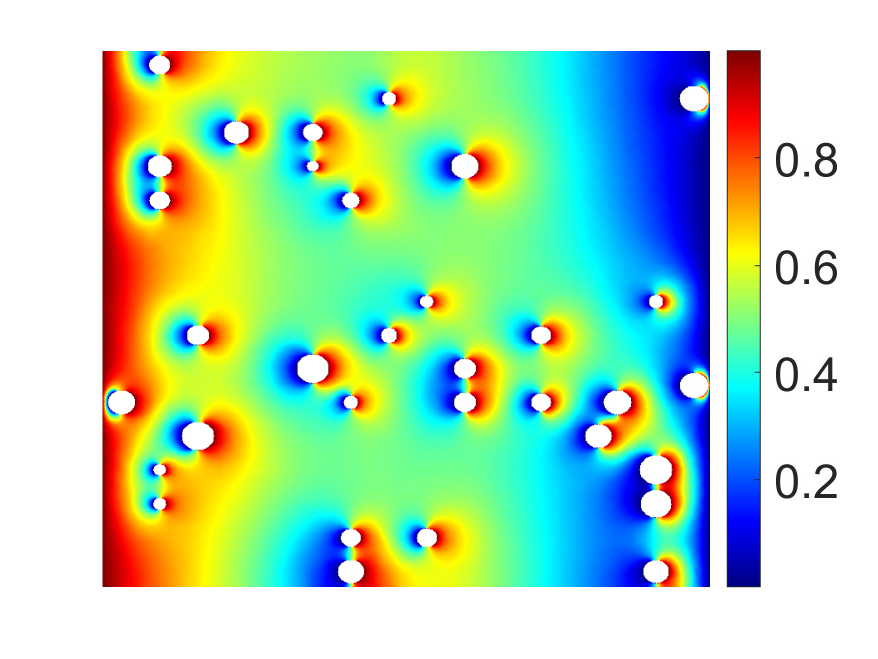}
\caption{}
\end{subfigure}
\caption{Numerical solutions for pressure in Model 1. (a) Reference solution; (b) Uniform enrichment online solution with $L=3$ and \texttt{Iter}=5;
(c) Adaptive enrichment online solution with $L=3$, $\theta=0.6$, and \texttt{Iter}=12.}
\label{fig:ex1_phms}
\end{figure}

\begin{figure}[htbp]
\centering
\begin{subfigure}[b]{0.32\textwidth}
\includegraphics[width=\linewidth]{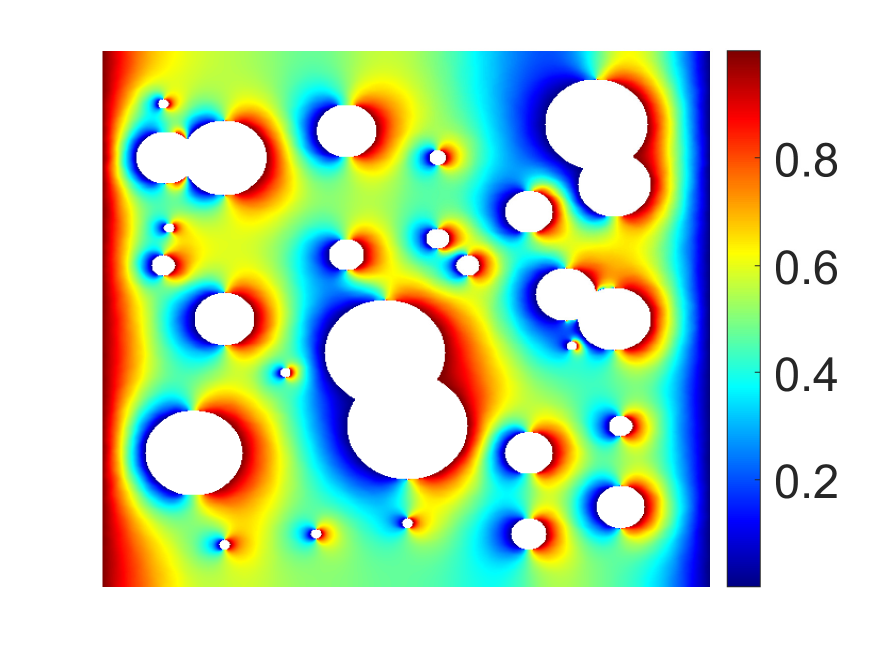}
\caption{}
\end{subfigure}
\begin{subfigure}[b]{0.32\textwidth}
\includegraphics[width=\linewidth]{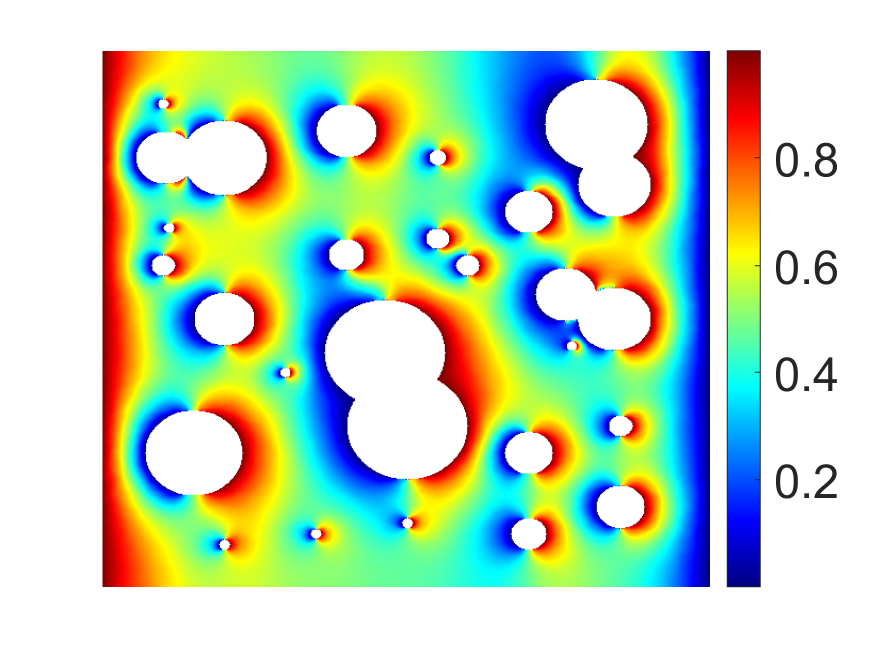}
\caption{}
\end{subfigure}
\begin{subfigure}[b]{0.32\textwidth}
\includegraphics[width=\linewidth]{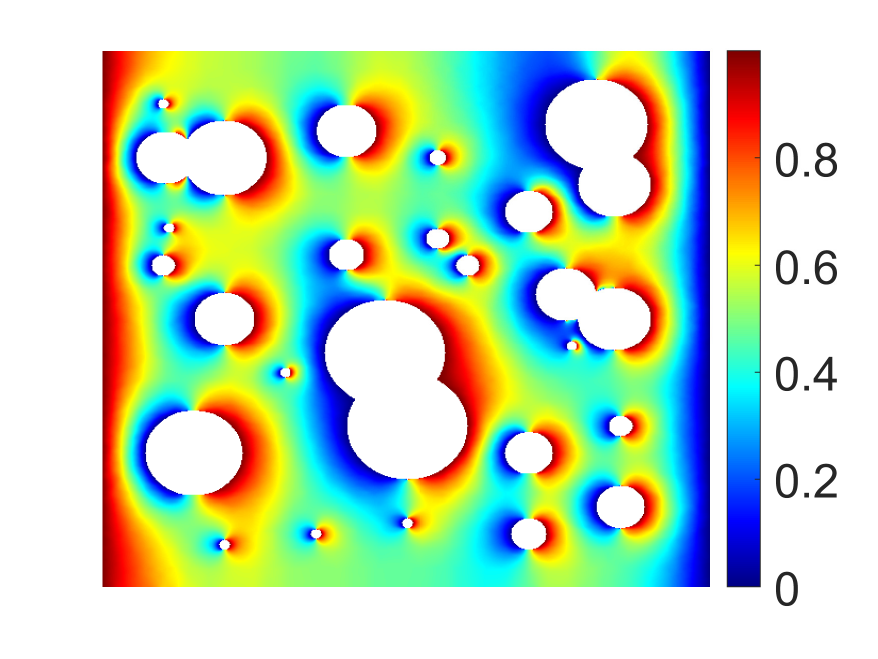}
\caption{}
\end{subfigure}
\caption{Numerical solutions for pressure in Model 2. (a) Reference solution; (b) Uniform enrichment online solution with $L=3$ and \texttt{Iter}=5;
(c) Adaptive enrichment online solution with $L=3$, $\theta=0.6$, and \texttt{Iter}=12.}
\label{fig:ex2_phms}
\end{figure}

\begin{figure}[htbp]
\centering
\begin{subfigure}[b]{0.32\textwidth}
\includegraphics[width=\linewidth]{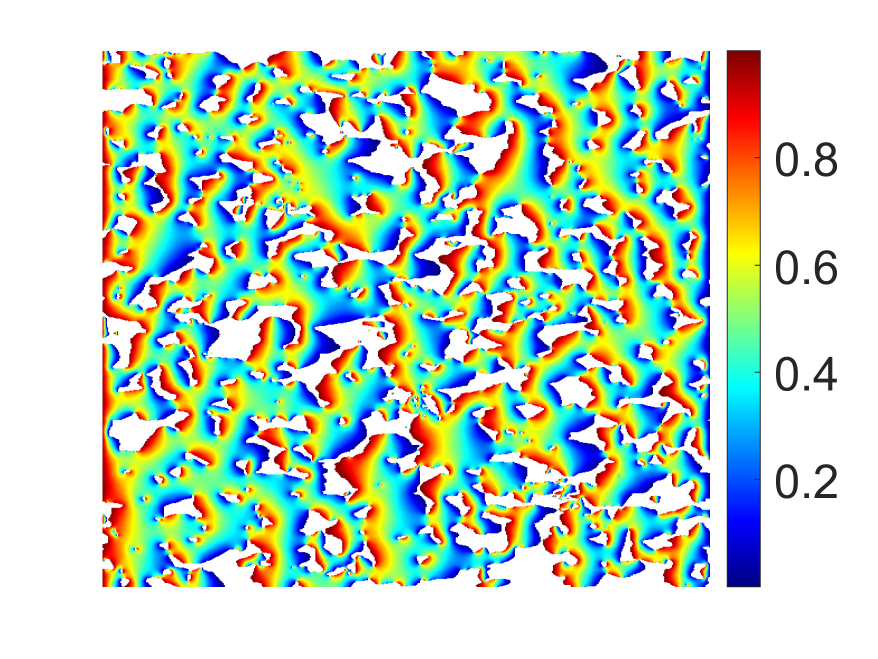}
\caption{}
\end{subfigure}
\begin{subfigure}[b]{0.32\textwidth}
\includegraphics[width=\linewidth]{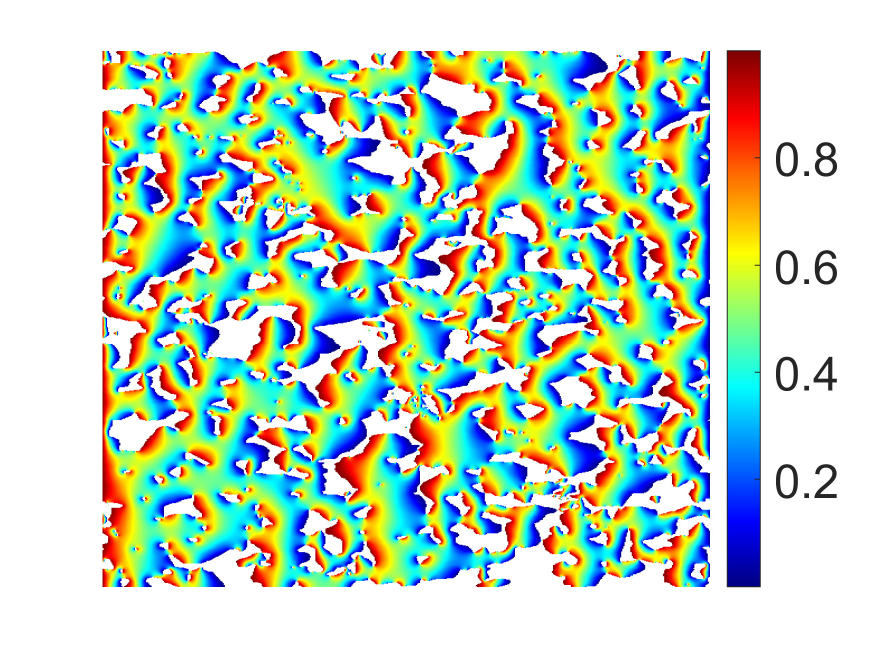}
\caption{}
\end{subfigure}
\begin{subfigure}[b]{0.32\textwidth}
\includegraphics[width=\linewidth]{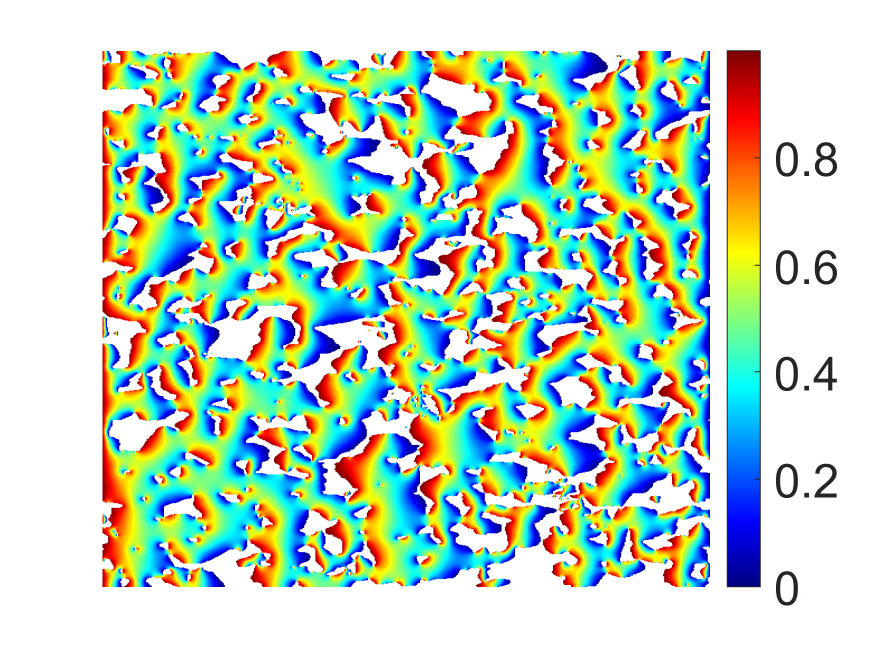}
\caption{}
\end{subfigure}
\caption{Numerical solutions for pressure in Model 3. (a) Reference solution; (b) Uniform enrichment online solution with $L=3$ and \texttt{Iter}=10;
(c) Adaptive enrichment online solution with $L=3$, $\theta=0.8$, and \texttt{Iter}=12.}
\label{fig:ex3_phms}
\end{figure}

\section{Conclusions} \label{sec:conclusions}
In this work, we propose an efficient multiscale method within the framework of the Generalized Multiscale Finite Element Method (GMsFEM) for solving the Darcy equation in perforated domains. To simplify the underlying mixed formulation, we apply the trapezoidal rule within the lowest-order Raviart–Thomas (RT0) finite element space, resulting in a diagonal velocity mass matrix. This allows us to eliminate the velocity variable and reformulate the original mixed velocity-pressure system as a pressure-only problem.
We construct two types of multiscale spaces for the pressure variable: an \emph{offline} space and an \emph{online} space. The offline space is generated via local spectral decomposition, effectively capturing the influence of heterogeneity in both the permeability field and geometric structures. The online space is adaptively enriched by solving local problems driven by residual information, thereby incorporating global features into the local approximation.
A rigorous numerical analysis is provided to support the convergence properties of the proposed method. Numerical experiments on representative benchmark problems demonstrate the efficiency and accuracy of the approach, and confirm the consistency between theoretical predictions and observed performance.

\bibliographystyle{abbrv}
\bibliography{references}
\end{document}